\documentclass{amsproc}
\usepackage{amssymb}
\usepackage{amsmath}
\usepackage{amsfonts}

\theoremstyle{plain}

\newtheorem{corollary}{Corollary}

\newtheorem{definition}{Definition}

\newtheorem{lemma}{Lemma}
\newtheorem{notation}{Notation}

\newtheorem{proposition}{Proposition}
\newtheorem{remark}{Remark}

\newtheorem{theorem}{Theorem}
\numberwithin{equation}{section}
\input{tcilatex}

\begin{document}
\title[Nonlinear degenerating parabolic equation \& embedding theorems]{On
Nonlinear Parabolic Equation in Nondivergent Form with Implicit Degeneration
and Embedding Theorems }
\author{Kamal N. Soltanov}
\address{{\small Department of Mathematics, Faculty of Sciences, Hacettepe
University Beytepe, Ankara, TR-06532, TURKEY}}
\email{{\small soltanov@hacettepe.edu.tr ; sultan\_kamal@hotmail.com }}
\urladdr{}
\author{Mahmud A. Ahmadov}
\curraddr{Department of Mathematics, Holyoke Community College, 303
Homestead Ave., Holyoke, MA, 01040 , USA }
\email{mahmadov@hcc.edu}
\urladdr{}
\date{}

\begin{abstract}
The mixed problem for the implicit degenerating nonlinear parabolic equation
is considered, and the solvability and behavior of solutions of this problem
are studied. Furthermore, some classes of function spaces and their
relations with Sobolev spaces are investigated, embedding and compactness
theorems for these spaces are proved
\end{abstract}

\maketitle

\section{Introduction}

\smallskip Consider the following problem 
\begin{equation}
\frac{\partial u}{\partial t}-\left\vert u\right\vert ^{\rho }\Delta
u+b_{0}\left\vert u\right\vert ^{\mu +1}=h\left( t,x\right) ,\quad \left(
t,x\right) \in Q_{T}\equiv \left( 0,T\right) \times \Omega ,  \label{Eqn 1.1}
\end{equation}%
\begin{equation}
u\left( 0,x\right) =0,\quad x\in \Omega \subset R^{n},\quad n\geq 1,
\label{Eqn 1.2}
\end{equation}%
\begin{equation}
u\left( t,x\right) \left\vert \ _{\Gamma }\right. =0,\quad \Gamma \equiv 
\left[ 0,T\right] \times \partial \Omega ,\ T>0,  \label{Eqn 1.3}
\end{equation}%
Here $\Omega $ is a bounded domain with sufficiently smooth boundary $%
\partial \Omega $ (for example, $\partial \Omega \in C^{1}$), $\Delta \equiv 
\overset{n}{\underset{i=1}{\sum }}\frac{\partial ^{2}}{\partial x_{i}^{2}}$
is a Laplacian, $\rho >0,\mu \geq 0$, $b_{0}\in R^{1}$ are some numbers, $%
h\left( t,x\right) $ is a certain function.{}

The equation (\ref{Eqn 1.1}) describes the behavior of a flow on a boundary
layer (see, \cite{O, Wa, H}) and is also called Prandtl-von Mises type
equation. The solvability of such type of equations and the behavior of
their solutions are considered in many works (for example, \cite{FMcL, LDP,
S1, S2, S3, TI, Wa, Wig, Win} and references therein). In one-dimensional
case, the existence of solution of the considered equation, and functional
spaces where the solution belongs to are obtained in \cite{S1, S5} (see also
references in \cite{S3}).

The main point of this work is considering the problem (\ref{Eqn 1.1}) - (%
\ref{Eqn 1.3}) in $n$-dimensional case without additional conditions.
Namely, the existence theorem is proved; spaces generated by the considered
problem, their properties (particularly, some smoothness results of
solutions are obtained as corrolaries of proved embedding theorems) and the
behavior of solution are studied.

Boundary value problems often lead to study of functional spaces related to
the considered problems directly. More precisely, mentioned spaces are
domains of operators generated by boundary value problems. For instance, we
can say that the Sobolev spaces and their different generalizations appear
while studying boundary value problems for the linear differential equations.

Unlike linear boundary value problems, in nonlinear cases sets generated by
problems, i.e. the domains of the corresponding operators, roughly speaking,
are subsets of linear spaces, but not possessing the linear structure.
Therefore, in the beginning we would be concentrated on investigation of
these infinity dimensional manifolds.

\section{Existence Theorem}

Define the following function space: 
\begin{equation}
\mathbf{P}_{1,p,q}\left( Q_{T}\right) \equiv W_{q}^{1}\left( 0,T;L_{q}\left(
\Omega \right) \right) \cap L_{p}\left( 0,T;\overset{0}{S}_{\Delta ,\rho
,2}\left( \Omega \right) \right) ,  \label{Eqn 2.1}
\end{equation}%
where $p,q\geq 1$, $\rho \geq 0$ are some numbers, $W_{q}^{1}\left(
0,T;L_{q}\left( \Omega \right) \right) $ is a vector Sobolev space, and for
functions $u:\Omega \longrightarrow R^{1}$

\begin{equation}
\overset{0}{S}_{\Delta ,\rho ,2}\left( \Omega \right) \equiv \left\{ u\in
L_{1}\left( \Omega \right) \left\vert \left[ u\right] _{S_{\Delta ,\rho
,2}}^{\rho +2}\equiv \underset{\Omega }{\int }\left\vert u\right\vert ^{\rho
}\left\vert \Delta u\right\vert ^{2}dx\right. <+\infty ,\ u\left( x\right)
\left\vert \ _{\partial \Omega }\right. =0\right\} .  \label{Eqn 2.2}
\end{equation}%
and for functions $u:Q_{T}\longrightarrow R^{1}$ 
\begin{equation*}
L_{p}\left( 0,T;\overset{0}{S}_{\Delta ,\rho ,2}\left( \Omega \right)
\right) \equiv \left\{ u\in L_{1}\left( Q_{T}\right) \left\vert ~\left[ u%
\right] _{L\left( S_{\Delta ,\rho ,2}\right) }^{p}\equiv \underset{0}{%
\overset{T}{\int }}\left[ u\right] _{S_{\Delta ,\rho ,2}}^{p}dt\right.
<+\infty ,\right. \ 
\end{equation*}%
\begin{equation}
\left. u\left( t,x\right) \left\vert \ _{\left[ 0,T\right] \times \partial
\Omega }\right. =0\right\} .  \label{Eqn 2.3}
\end{equation}

Our main result on solvability of the problem (\ref{Eqn 1.1}) - (\ref{Eqn
1.3}) is

\begin{theorem}
\label{Theorem 2.1}Let $\rho >0$ , $\min \left\{ 0,\frac{\rho }{2}-1\right\}
\leq \mu <\rho \leq 2$ or\ $\frac{\rho }{2}-1\leq \mu <\rho $ and $b_{0}\in
R^{1}$. Then, for any $h\in L_{2}\left( 0,T;\overset{0}{W}\text{\/}%
_{2}^{1}\left( \Omega \right) \right) $ the problem (\ref{Eqn 1.1}) - (\ref%
{Eqn 1.3}) is solvable in $\mathbf{P}\left( Q\right) $ $\equiv $ $\mathbf{P}%
_{1,p,q}\left( Q_{T}\right) \cap \left\{ u\left( t,x\right) \left\vert \
u\left( 0,x\right) =0\right. \right\} $, where $p=\rho +2$, $q=p%
{\acute{}}%
=\frac{\rho +2}{\rho +1}$.
\end{theorem}

The proof is based on a general result (Theorem \ref{Theorem 2.2}) that is
given below.

Let $X$ and $Y$ be Banach spaces with duals $X^{\ast }$ and $Y^{\ast }$
respectively, $Y$ is a reflexive Banach space, $\mathcal{M}_{0}\subseteq X$
be a weakly complete "reflexive" pn-space (see, Appendix A or [S3, S5]), $%
X_{0}\subseteq \mathcal{M}_{0}\cap Y$ is a separable vector topological
space such that $\overline{X_{0}}^{\mathcal{M}_{0}}\equiv \mathcal{M}_{0}$, $%
\overline{X_{0}}^{Y}\equiv Y$ and\newline

\textit{i}) $f:\underset{0}{\mathbf{P}}$\/$_{1,p,q}\left( 0,T;\mathcal{M}%
_{0},Y\right) \rightarrow L_{q}\left( 0,T;Y\right) $ is a weakly compact
(weakly continuous) mapping, where 
\begin{equation*}
\underset{0}{\mathbf{P}}\/_{1,p,q}\left( 0,T;\mathcal{M}_{0},Y\right) \equiv
L_{p}\left( 0,T;\mathcal{M}_{0}\right) \cap W_{q}^{1}\left( 0,T;Y\right)
\cap \left\{ x\left( t\right) \left\vert \ x\left( 0\right) =0\right.
\right\} ,
\end{equation*}%
$1<\max \{q,q^{\prime }\}\leq p<\infty $, $q^{\prime }=\frac{q}{q-1}$;%
\newline

(\textit{ii}) there is a linear continuous operator $L:W_{p_{0}}^{s}\left(
0,T;X_{0}\right) \rightarrow W_{p_{0}}^{s}\left( 0,T;Y^{\ast }\right) $, $%
s\geq 0$, $p_{0}\geq 1$ such that $L$ commutes with $\frac{\partial }{%
\partial t}$ and the conjugate operator $L^{\ast }$ has $ker(L^{\ast
})=\left\{ 0\right\} $;

\textit{(iii) }there exist a continuous function $\varphi :R_{+}^{1}\cup
\left\{ 0\right\} \longrightarrow R_{+}^{1}\cup \left\{ 0\right\} $ and
numbers $\tau _{0}\geq 0$ and $\tau _{1}>0$ such that $\varphi (r)$ is not
decreasing for $\tau \geq \tau _{0}$,\ $\varphi \left( \tau _{1}\right) >0$
and for any $x\in L_{p}\left( 0,T;X_{0}\right) $ operators $f$ and $L$
satisfy the inequality 
\begin{equation*}
\underset{0}{\overset{T}{\int }}\langle f(t,x\left( t\right) ),Lx\left(
t\right) \rangle dt\geq \varphi \left( \lbrack x]_{L_{p}\left( \mathcal{M}%
_{0}\right) }\right) [x]_{L_{p}\left( \mathcal{M}_{0}\right) };
\end{equation*}

\textit{(iv) }there exist a linear bounded operator $L_{0}:X_{0}\rightarrow
Y $ and constants $C_{0}>0$, $C_{1},C_{2}\geq 0$, $\nu >1$ such that the
inequalities 
\begin{eqnarray*}
\underset{0}{\overset{T}{\int }}\langle \xi \left( t\right) ,L\xi \left(
t\right) \rangle dt &\geq &C_{0}\left\Vert L_{0}\xi \right\Vert
_{L_{q}\left( 0,T;Y\right) }^{\nu }-C_{2}, \\
\underset{0}{\overset{t}{\int }}\langle \frac{dx}{d\tau },Lx\left( \tau
\right) \rangle d\tau &\geq &C_{1}\left\Vert L_{0}x\right\Vert _{Y}^{\nu
}\left( t\right) -C_{2},\quad a.e.\ t\in \left[ 0,T\right]
\end{eqnarray*}%
hold for arbitrary $x\in W_{p}^{1}\left( 0,T;X_{0}\right) $ and $\xi \in
L_{p}\left( 0,T;X_{0}\right) $.

\begin{theorem}
\label{Theorem 2.2}Assume that conditions (i) - (iv) are fulfiled. Then, for
any $y\in G\subseteq L_{q}\left( 0,T;Y\right) $, $G\equiv $ $\underset{r\geq
\tau _{1}}{\cup }G_{r}$:\newline
\begin{equation*}
G_{r}\equiv \left\{ y\in L_{q}\left( 0,T;Y\right) \left\vert \underset{0}{%
\overset{T}{\int }}\left\vert \langle y\left( t\right) ,Lx\left( t\right)
\rangle \right\vert ~dt\leq \underset{0}{\overset{T}{\int }}\langle
f(t,x\left( t\right) ),Lx\left( t\right) \rangle dt\right. -c,\right.
\end{equation*}%
\begin{equation*}
\left. for\text{ }all\text{ }x\in L_{p}\left( 0,T;X_{0}\right) ,\ \left[ x%
\right] _{L_{p}\left( 0,T;\mathcal{M}_{0}\right) }=r\right\} ,\
C_{2}<c<\infty
\end{equation*}%
the Cauchy problem 
\begin{equation}
\frac{dx}{dt}+f(t,x\left( t\right) )=y\left( t\right) ,\quad y\in
L_{q}\left( 0,T;Y\right) ;\quad x\left( 0\right) =0  \label{Eqn 2.4}
\end{equation}

is solvable in $\underset{0}{\mathbf{P}}$\/$_{1,p,q}\left( 0,T;\mathcal{M}%
_{0},Y\right) $ in the following sense 
\begin{equation*}
\underset{0}{\overset{T}{\int }}\left\langle \frac{dx}{dt}+f(t,x\left(
t\right) ),y^{\ast }\left( t\right) \right\rangle dt=\underset{0}{\overset{T}%
{\int }}\left\langle y\left( t\right) ,y^{\ast }\left( t\right)
\right\rangle dt,\quad \forall y^{\ast }\in L_{q\prime }\left( 0,T;Y^{\ast
}\right) .
\end{equation*}
\end{theorem}

The proof of this result is presented in Appendix C (one can also refer to
proofs of similar theorems in \cite{S3, S5}). The next proposition follows
immediately from the last theorem.

\begin{corollary}
\label{Corollary 2.1}Under assumptions of Theorem \ref{Theorem 2.2} the
problem (\ref{Eqn 2.4}) is solvable in $\underset{0}{\mathbf{P}}$\/$%
_{1,p,q}\left( 0,T;\mathcal{M}_{0},Y\right) $ for any $y\in L_{q}\left(
0,T;Y\right) $ satisfying the condition: there is $r>0$ such that the
inequality 
\begin{equation*}
\left\Vert y\right\Vert _{L_{q}\left( 0,T;Y\right) }\leq \varphi \left(
\lbrack x]_{L_{p}\left( 0,T;\mathcal{M}_{0}\right) }\right)
\end{equation*}%
holds for any $x\in L_{p}\left( 0,T;X_{0}\right) $ with $[x]_{L_{p}\left( 
\mathcal{M}_{0}\right) }=r$. Furthermore, if $\varphi \left( \tau \right)
\nearrow \infty $ as $\tau \nearrow \infty $ then the problem (\ref{Eqn 2.4}%
) is solvable in $\underset{0}{\mathbf{P}}$\/$_{1,p,q}\left( 0,T;\mathcal{M}%
_{0},Y\right) $ for any $y\in L_{q}\left( 0,T;Y\right) $ satisfying the
inequality 
\begin{equation*}
\sup \left\{ \frac{1}{[x]_{L_{p}\left( 0,T;\mathcal{M}_{0}\right) }}\underset%
{0}{\overset{T}{\int }}\langle y\left( t\right) ,Lx\left( t\right) \rangle
~dt\ \left\vert \ x\in L_{p}\left( 0,T;X_{0}\right) \right. \right\} <\infty
.
\end{equation*}
\end{corollary}

\section{Embedding Theorems on pn-Spaces}

In this section we introduce and investigate properties of a class of
nonlinear function spaces (pn-spaces) that are connected to the considered
problem directly. These spaces are necessary in application of Theorem \ref%
{Theorem 2.2} (and Corollary \ref{Corollary 2.1}) to the considered problem.

Consider the following function spaces (class of functions $u:\Omega
\longrightarrow R$) 
\begin{equation}
S_{1,\alpha ,\beta }\left( \Omega \right) \equiv \left\{ u\in L_{1}\left(
\Omega \right) ~\left\vert ~\left[ u\right] _{S_{1}}^{\alpha +\beta }\equiv 
\underset{\Omega }{\int }~\left[ \left\vert u\right\vert ^{\alpha +\beta
}+\left\vert u\right\vert ^{\alpha }\left\vert \nabla u\right\vert ^{\beta }%
\right] dx<\infty ,\right. \right\} ,  \label{Eqn 3.1}
\end{equation}%
\begin{equation}
S_{\Delta ,\alpha ,\beta }\left( \Omega \right) \equiv \left\{ u\in
L_{1}\left( \Omega \right) ~\left\vert ~\left[ u\right] _{S_{\Delta
}}^{\alpha +\beta }\equiv \left[ u\right] _{S_{1}}^{\alpha _{1}+\beta _{1}}+%
\underset{\Omega }{\int }~\left\vert u\right\vert ^{\alpha }\left\vert
\Delta u\right\vert ^{\beta }dx<\infty ,\right. \right\} ,  \label{Eqn 3.2}
\end{equation}%
where $\alpha \geq 0,$ $\frac{\alpha _{1}}{\beta _{1}}>-1$, $\beta ,\beta
_{1}\geq 1$ and $\alpha _{1}+\beta _{1}=\alpha +\beta $. Here and hereafter
we assume $\beta >1$. Further, we consider the case $\frac{\alpha }{\beta }%
>-1$, $\beta >1$, $\alpha >\beta -1$, as well.

Also, consider the following spaces of functions $u:Q_{T}\longrightarrow
R^{1}$ 
\begin{equation}
L_{p}\left( 0,T;S_{1,\alpha ,\beta }\left( \Omega \right) \right) \equiv
\left\{ u\in L_{1}\left( \Omega \right) ~\left\vert ~\left[ u\right]
_{L\left( S_{1}\right) }^{p}\equiv \underset{0}{\overset{T}{\int }}~\left[ u%
\right] _{S_{1}}^{p}dt<\infty ,\right. \right\} ,  \label{Eqn 3.3}
\end{equation}%
\begin{equation}
P_{p_{0},p_{1}}\left( 0,T;S_{\Delta ,\alpha ,\beta }\left( \Omega \right)
;X\right) \equiv W_{p_{0}}^{1}\left( 0,T;X\right) \cap L_{p_{1}}\left(
0,T;S_{\Delta ,\alpha ,\beta }\left( \Omega \right) \right) ,\quad
\label{Eqn 3.4}
\end{equation}%
where $p,p_{0},p_{1},\beta >1$, $\alpha \geq 0$ and $X$ is a Banach space.
Particularly, $X$ can be choosen in such a way that $L_{p_{0}}\left( \Omega
\right) \subseteq X$ for some $p_{0}\geq 1$.

The space $L_{p_{1}}\left( 0,T;S_{\Delta ,\alpha ,\beta }\left( \Omega
\right) \right) $ is defined as $L_{p}\left( 0,T;S_{1,\alpha ,\beta }\left(
\Omega \right) \right) $ by using (\ref{Eqn 3.2}) instead of (\ref{Eqn 3.1}).

The equivalency 
\begin{equation*}
\mathcal{M}_{\eta ,W_{\beta }^{1}\left( \Omega \right) }\equiv \left\{ u\in
L_{1}\left( \Omega \right) ~\left\vert ~\eta \left( u\right) \in W_{\beta
}^{1}\left( \Omega \right) ,\ \eta \left( u\right) \equiv \left\vert
u\right\vert ^{\frac{\alpha }{\beta }}u\right. \right\} \equiv S_{1,\alpha
,\beta }\left( \Omega \right)
\end{equation*}%
that express relations between $W_{\beta }^{1}\left( \Omega \right) $ and $%
S_{1,\alpha ,\beta }\left( \Omega \right) $\ follows immediately from (\ref%
{Eqn 3.1}). Indeed, it is enough to note that $\eta \left( u\right) \equiv
\left\vert u\right\vert ^{\frac{\alpha }{\beta }}u=v$ $\Longleftrightarrow
u=\left\vert v\right\vert ^{\frac{-\alpha }{\alpha +\beta }}v\equiv \eta
^{-1}\left( v\right) $.

Taking the last equivalency and definition (\ref{Eqn 3.2}) of the space $%
S_{\Delta ,\alpha ,\beta }\left( \Omega \right) $ into account we get

\begin{equation}
S_{\Delta ,\alpha ,\beta }\left( \Omega \right) \equiv \mathcal{M}_{\eta
,W_{\beta _{1}}^{1}\left( \Omega \right) }\cap \left\{ u~\left\vert \
\left\vert u\right\vert ^{\frac{\alpha }{\beta }}\Delta u\in L_{\beta
}\left( \Omega \right) ,\right. ~\right\} .  \label{Eqn 3.5}
\end{equation}

In our next step we are going to express the relations between the second
order Sobolev spaces and $S_{\Delta ,\alpha ,\beta }\left( \Omega \right) .$
To this end we use a few auxilary results.

The following equality will be used in our discussion. Let's put $\eta
\left( u\right) \equiv \left\vert u\right\vert ^{\frac{\alpha }{\beta }%
}u\equiv v.$ Then 
\begin{equation*}
\Delta v\equiv \left( \Delta \circ ~\eta \right) \left( u\right) \equiv
\Delta \eta \left( u\right) \equiv \Delta \left( \left\vert u\right\vert ^{%
\frac{\alpha }{\beta }}u\right) =\nabla \cdot \left( \frac{\alpha +\beta }{%
\beta }~\left\vert u\right\vert ^{\frac{\alpha }{\beta }}~\nabla u\right) =
\end{equation*}%
\begin{equation}
=\frac{\alpha +\beta }{\beta }\left\vert u\right\vert ^{\frac{\alpha }{\beta 
}}~\Delta u+\frac{\alpha \left( \alpha +\beta \right) }{\beta ^{2}}%
~\left\vert u\right\vert ^{\frac{\alpha }{\beta }-2}~u~\left\vert \nabla
u\right\vert ^{2}.  \label{Eqn 3.6}
\end{equation}

\begin{proposition}
\label{Proposition 3.1}Let $\alpha >-1,\beta \geq \beta _{0}\geq 0$, $\beta
\geq 1$ be some numbers, $\beta _{0}+\beta \geq 2$ and $\Omega \subset
R^{n}, $ $n\geq 1,$ be a bounded domain with sufficiently smooth boundary $%
\partial \Omega $. Then the inequality 
\begin{equation}
\underset{\Omega }{\int }\left\vert u\right\vert ^{\alpha }\left\vert \nabla
u\right\vert ^{\beta _{0}+\beta }dx\leq c\left( \varepsilon \right) \underset%
{i=1}{\overset{n}{\sum }}~\underset{\Omega }{\int }\left\vert u\right\vert
^{\alpha +\beta _{0}}\left\vert D_{i}^{2}u\right\vert ^{\beta
}dx+\varepsilon _{1}\kappa \left( \beta -\beta _{0}\right) \underset{\Omega }%
{\int }\left\vert u\right\vert ^{\alpha +\beta _{0}+\beta }dx
\label{Eqn 3.7}
\end{equation}%
holds for any $u\in C^{2}\left( \Omega \right) \cap C_{0}^{1}\left( 
\overline{\Omega }\right) $, where $\varepsilon >0$, $\varepsilon
_{1}=\varepsilon _{1}\left( \varepsilon \right) $ are some numbers, $\kappa
\left( s\right) =1$ if $s>0$, and $\kappa \left( s\right) =0$ if $s=0$.
\end{proposition}

\begin{proof}
We have 
\begin{equation*}
\underset{\Omega }{\int }\left\vert u\right\vert ^{\alpha }\left\vert \nabla
u\right\vert ^{\beta _{0}+\beta }dx\leq c\underset{i=1}{\overset{n}{\sum }}%
\underset{\Omega }{\int }\left\vert u\right\vert ^{\alpha }\left\vert
D_{i}u\right\vert ^{\beta _{0}+\beta }dx=
\end{equation*}%
\begin{equation*}
-c_{1}\underset{i=1}{\overset{n}{\sum }}~\underset{\Omega }{\int }\left\vert
u\right\vert ^{\alpha }u\left\vert D_{i}u\right\vert ^{\beta _{0}+\beta
-2}D_{j}^{2}udx\leq
\end{equation*}%
Rewriting the expression under the integral in the following form

\begin{equation*}
\left( \left\vert u\right\vert ^{\alpha -\frac{\alpha +\beta _{0}}{\beta }%
-\alpha \frac{\beta _{0}+\beta -2}{\beta _{0}+\beta }}u\right) \left(
\left\vert u\right\vert ^{\alpha \frac{\beta _{0}+\beta -2}{\beta _{0}+\beta 
}}\left\vert D_{i}u\right\vert ^{\beta +\beta _{0}-2}\right) \left(
\left\vert u\right\vert ^{\frac{\alpha +\beta _{0}}{\beta }%
}D_{i}^{2}u\right) \text{ if\ }\beta >\beta _{0}
\end{equation*}

or 
\begin{equation*}
\left( \left\vert u\right\vert ^{\frac{\alpha }{\beta ^{\prime }}%
}u\left\vert D_{i}u\right\vert ^{2\beta -2}\right) \left( \left\vert
u\right\vert ^{\frac{\alpha }{\beta }}D_{i}^{2}u\right) \text{ if\ }\beta
=\beta _{0}
\end{equation*}

and applying Young's inequality with exponents

\begin{equation*}
p_{0}=\frac{\beta \left( \beta _{0}+\beta \right) }{\beta -\beta _{0}},p_{1}=%
\frac{\beta _{0}+\beta }{\beta _{0}+\beta -2},\ p_{2}=\beta \text{ \ if}\
\beta >\beta _{0}
\end{equation*}

or

\begin{equation*}
p_{0}=\beta ^{\prime },p_{1}=\beta \text{ if\ }\beta =\beta _{0}\text{, }%
\frac{1}{\beta }+\frac{1}{\beta ^{\prime }}=1
\end{equation*}

we get%
\begin{equation*}
\leq \varepsilon \underset{i=1}{\overset{n}{\sum }}~\underset{\Omega }{\int }%
\left[ \kappa \left( \beta -\beta _{0}\right) \left\vert u\right\vert
^{\alpha +\beta _{0}+\beta }+\left\vert u\right\vert ^{\alpha }\left\vert
D_{i}u\right\vert ^{\beta _{0}+\beta }\right] dx+
\end{equation*}%
\begin{equation*}
c\left( \varepsilon \right) \underset{i=1}{\overset{n}{\sum }}~\underset{%
\Omega }{\int }\left\vert u\right\vert ^{\alpha +\beta _{0}}\left\vert
D_{i}^{2}u\right\vert ^{\beta }dx\leq
\end{equation*}

or%
\begin{equation*}
\varepsilon _{1}\kappa \left( \beta -\beta _{0}\right) \underset{\Omega }{%
\int }\left\vert u\right\vert ^{\alpha +\beta _{0}+\beta }dx+\varepsilon _{2}%
\underset{\Omega }{\int }\left\vert u\right\vert ^{\alpha }\left\vert \nabla
u\right\vert ^{\beta _{0}+\beta }dx+
\end{equation*}%
\begin{equation}
c_{4}\left( \varepsilon _{1},\varepsilon _{2}\right) \underset{i=1}{\overset{%
n}{\sum }}\underset{\Omega }{\int }\left\vert u\right\vert ^{\alpha +\beta
_{0}}\left\vert D_{i}^{2}u\right\vert ^{\beta }dx  \label{Eqn 3.8}
\end{equation}%
The second term in (\ref{Eqn 3.8}) is obtained by using the equivalency 
\begin{equation}
\underset{\Omega }{\int }\left\vert u\right\vert ^{\alpha }\underset{i=1}{%
\overset{n}{\sum }}\left\vert D_{i}u\right\vert ^{\beta _{0}+\beta }dx\leq 
\underset{\Omega }{\int }\left\vert u\right\vert ^{\alpha }\left\vert \nabla
u\right\vert ^{\beta _{0}+\beta }dx\leq n\underset{\Omega }{\int }\left\vert
u\right\vert ^{\alpha }\underset{i=1}{\overset{n}{\sum }}\left\vert
D_{i}u\right\vert ^{\beta _{0}+\beta }dx.  \label{Eqn 3.9}
\end{equation}%
Note that the first term of (\ref{Eqn 3.8}) vanishes if $\beta =\beta _{0}$.
Obtained inequalities prove the statement of the proposition.
\end{proof}

\begin{remark}
\label{Remark 3.1}It is not difficult to see that if $\alpha +\beta
_{0}+\beta >1$, $\beta _{0}\geq 0$, $\beta _{1}\geq 1$ then 
\begin{equation}
\underset{\Omega }{\int }\left\vert u\right\vert ^{\alpha +\beta _{0}+\beta
}dx\leq c\underset{\Omega }{\int }\left\vert u\right\vert ^{\alpha +\beta
_{0}}\left\vert \nabla u\right\vert ^{\beta }dx\text{ \ or }\underset{\Omega 
}{\int }\left\vert u\right\vert ^{\alpha +\beta _{0}+\beta }dx\leq c\underset%
{\Omega }{\int }\left\vert u\right\vert ^{\alpha }\left\vert \nabla
u\right\vert ^{\beta _{0}+\beta }dx  \label{Eqn 3.10}
\end{equation}%
and if $1\leq \alpha _{0}+\beta _{0}\leq \alpha _{1}+\beta _{1}$, $1\leq
\beta _{0}\leq \beta _{1}$, $\alpha _{0}\beta _{1}\geq \alpha _{1}\beta _{0}$
then%
\begin{equation}
\underset{\Omega }{\int }\left\vert u\right\vert ^{\alpha _{0}}\left\vert
\nabla u\right\vert ^{\beta _{0}}dx\leq c\underset{\Omega }{\int }\left\vert
u\right\vert ^{\alpha _{1}}\left\vert \nabla u\right\vert ^{\beta
_{1}}dx+c_{1}  \label{Eqn 3.11}
\end{equation}%
hold for any $u\in C_{0}^{1}\left( \Omega \right) $, where 
\begin{equation*}
c=c\left( \alpha ,\beta _{0},\beta ,mes\ \Omega \right) >0,\
c_{1}=c_{1}\left( \alpha _{0},\beta _{0},\alpha _{1},\beta _{1},mes\ \Omega
\right) \geq 0,
\end{equation*}%
Moreover, if $\alpha _{0}+\beta _{0}=\alpha _{1}+\beta _{1}$ then $c_{1}=0$.
\end{remark}

\begin{proposition}
\label{Proposition 3.2}Let $\alpha >-1$, $\beta \geq 1$ be some numbers, $%
\alpha +\beta \geq 2$ and $\Omega \subset R^{n},$ $n\geq 1,$ be a bounded
domain with sufficiently smooth boundary $\partial \Omega $. Then the
inequality 
\begin{equation}
\underset{\Omega }{\int }\left\vert u\right\vert ^{\alpha +\beta }dx\leq c%
\underset{\Omega }{\int }\left\vert u\right\vert ^{\alpha }\left\vert \Delta
u\right\vert ^{\beta }dx.  \label{Eqn 3.12}
\end{equation}

holds for any $u\in C^{2}\left( \Omega \right) \cap C_{0}^{1}\left( 
\overline{\Omega }\right) $, where $c=c\left( \alpha ,\beta ,mes\ \Omega
\right) >0$.
\end{proposition}

\begin{proof}
Rewriting $\alpha +\beta $ as $\alpha +\beta -2+2=\alpha +\beta _{0}+\beta
_{1}$ with $\beta _{0}=\beta -2$ and $\beta _{1}=2$ and applying the first
one of inequalities (2.10) we get 
\begin{equation}
\underset{\Omega }{\int }\left\vert u\right\vert ^{\alpha +\beta }dx\equiv 
\underset{\Omega }{\int }\left\vert u\right\vert ^{\alpha +(\beta
-2)+2}dx\leq c\underset{\Omega }{\int }\left\vert u\right\vert ^{\alpha
+\beta -2}\left\vert \nabla u\right\vert ^{2}dx.  \label{Eqn 3.13}
\end{equation}

The right hand side of the last inequality is estimating as 
\begin{equation*}
\underset{\Omega }{\int }\left\vert u\right\vert ^{\alpha +\beta
-2}\left\vert \nabla u\right\vert ^{2}dx=\frac{1}{\alpha +\beta -1}\underset{%
\Omega }{\int }\nabla \left( \left\vert u\right\vert ^{\alpha +\beta
-2}u\right) \cdot \nabla udx=
\end{equation*}%
\begin{equation*}
-\frac{1}{\alpha +\beta -1}\underset{\Omega }{\int }\left\vert u\right\vert
^{\alpha +\beta -2}u\Delta udx\leq c\underset{\Omega }{\int }\left\vert
u\right\vert ^{\alpha +\beta -1}\left\vert \Delta u\right\vert dx=
\end{equation*}

\begin{equation*}
c\underset{\Omega }{\int }\left\vert u\right\vert ^{\alpha +\beta -1-\frac{%
\alpha }{\beta }}\left\vert u\right\vert ^{\frac{\alpha }{\beta }}\left\vert
\Delta u\right\vert dx
\end{equation*}

Now, applying the Young's inequality with exponents $\left( \beta ,\frac{%
\beta }{\beta -1}\right) $ and arbitrary $\varepsilon >0$ gives 
\begin{equation}
\leq c\left( \varepsilon \right) \underset{\Omega }{\int }\left\vert
u\right\vert ^{\alpha }\left\vert \Delta u\right\vert ^{\beta
}dx+\varepsilon \underset{\Omega }{\int }\left\vert u\right\vert ^{\alpha
+\beta }dx.  \label{Eqn 3.14}
\end{equation}

The inequality (\ref{Eqn 3.12}) follows from (\ref{Eqn 3.13}) taking (\ref%
{Eqn 3.14}) into considiration and making $\varepsilon $ sufficiently small.
\end{proof}

The following result is a special case of the main inequality (\ref{Eqn 3.22}%
)\footnote{%
For $n=1$ the similar results to results of this section was proved in the
earlier works (see, for example, [S1, S5]). Therefore, it is enough to
consider just dimension $n\geq 2$.}

\begin{lemma}
\label{Lemma 3.1}Let $\alpha >-1$, $\beta >\frac{n}{n-1}$ be some numbers, $%
\Omega \subset R^{n},$ $n\geq 2,$ be a bounded domain with sufficiently
smooth boundary $\partial \Omega $. Then the inequality 
\begin{equation}
\underset{\Omega }{\int }\left\vert u\right\vert ^{\alpha }\left\vert \nabla
u\right\vert ^{2\beta }dx\leq c_{1}\underset{\Omega }{\int }\left\vert
u\right\vert ^{\alpha +\beta }\left\vert \Delta u\right\vert ^{\beta
}dx+c_{2}\underset{\Omega }{\int }\left\vert u\right\vert ^{\alpha +2\beta
}dx.  \label{Eqn 3.15}
\end{equation}%
holds for any $u\in C^{2}\left( \Omega \right) \cap C_{0}^{1}\left( 
\overline{\Omega }\right) $, where $c=c\left( \alpha ,\beta \right) >0$.
\end{lemma}

\begin{proof}
\footnote{%
It should be noted that this approach of the proof is suggested by the
second author.}The proof of the inequality (\ref{Eqn 3.15}) is based on the
boundedness in the Lebesque space $L_{p}(\Omega )$ of the local
Hardy-Littlwood maximal function

\begin{equation*}
M_{\Omega }w\left( x\right) =\underset{0<r<dist(x,\partial \Omega )}{\sup }%
\frac{1}{\left\vert B_{r}\left( x\right) \right\vert }\underset{B_{r}\left(
x\right) }{\int }w~(y)dy;
\end{equation*}%
\begin{equation*}
\left\vert B_{r}\left( x\right) \right\vert \equiv \mu \left( B_{r}\left(
x\right) \right) \equiv \frac{\pi ^{\frac{n}{2}}}{\Gamma \left( \frac{n}{2}%
+1\right) }
\end{equation*}%
when $1<p<+\infty $ (see \cite{St}), the local spherical maximal function 
\begin{equation*}
\left( A_{r}w\right) \left( x\right) =\underset{0<r<dist(x,\partial \Omega )}%
{\sup }\underset{S_{1}\left( 0\right) }{\int }w(x+ry)~dS(y);\
S_{r}(x)=\partial B_{r}(x)
\end{equation*}%
when $p>\frac{n}{n-1},$ $n\geq 2$ (see \cite{St}), and on $L_{p}(\Omega )-$%
convergency of averages of a function to the function itself%
\begin{equation}
\underset{r\searrow 0}{\lim \text{ }}\underset{\Omega }{\int }\left\vert 
\frac{1}{\left\vert B_{r}\left( x\right) \right\vert }\underset{B_{r}\left(
x\right) }{\int }w(y)~dy-w(x)\right\vert ^{p}dx=0  \label{Eqn 3.16}
\end{equation}

Let's put $w(x)\equiv \left\vert u(x)\right\vert ^{\rho }\left\vert \nabla
u(x)\right\vert ^{2}$ for a function $u\in C^{2}\left( \Omega \right) \cap
C_{0}^{1}\left( \overline{\Omega }\right) $. Then, under the conditions of
Propositon \ref{Proposition 3.1} and boundedness of the local
Hardy-Littlwood maximal function, for $\rho =\frac{\alpha }{\beta }$ we have 
\begin{equation}
\underset{\Omega }{\int }\left( \frac{1}{\left\vert B_{r}\left( x\right)
\right\vert }\underset{B_{r}\left( x\right) }{\int }\left\vert u\right\vert
^{\rho }\left\vert \nabla u\right\vert ^{2}~dy\right) ^{\beta }dx\leq c%
\underset{\Omega }{\int }\left( \left\vert u\right\vert ^{p}\left\vert
\nabla u\right\vert ^{2}\right) ^{\beta }dx.  \label{Eqn 3.17}
\end{equation}%
Moreover, it is obvious that 
\begin{equation*}
\frac{1}{\left\vert B_{r}\left( x\right) \right\vert }\underset{S_{r}\left(
x\right) }{\int }\left\vert u\right\vert ^{\rho }u\frac{\partial u}{\partial
\nu }~dS\left( y\right) =
\end{equation*}

\begin{equation*}
\frac{1}{\left\vert B_{1}\left( 0\right) \right\vert r^{n}}\underset{%
S_{1}\left( 0\right) }{\int }\left\vert u(x+r\eta )\right\vert ^{\rho
}u(x+r\eta )\left( \nabla u\left( x+r\eta \right) \cdot \nu \right)
r~r^{n-1}dS\left( \eta \right) ,
\end{equation*}

Therefore, from the boundedness of a local spherical maximal function, we
have 
\begin{equation*}
\underset{\Omega }{\int }\left\vert \frac{1}{\left\vert S_{1}\left( 0\right)
\right\vert }\underset{S_{1}\left( 0\right) }{\int }\left\vert u(x+r\eta
)\right\vert ^{\rho }u(x+r\eta )\left( \nabla u\left( x+r\eta \right) \cdot
\nu \right) ds\left( \eta \right) \right\vert ^{\beta }dx\leq
\end{equation*}

\begin{equation}
c\underset{\Omega }{\int }\left( \left\vert u(x)\right\vert ^{\rho
+1}\left\vert \nabla u\left( x\right) \right\vert \right) ^{\beta }dx,
\label{Eqn 3.18}
\end{equation}

where the positive constant $c$ does not dependent on the function $u\left(
x\right) $.

According to (\ref{Eqn 3.6}) we have 
\begin{equation*}
\nabla \cdot \left( ~\left\vert u\right\vert ^{\rho }~u\nabla u\right)
=\left\vert u\right\vert ^{\rho }\Delta u+\left( \rho +1\right) ~\left\vert
u\right\vert ^{\rho }~\left\vert \nabla u\right\vert ^{2}
\end{equation*}%
Taking the integral of both sides of this equality on $B_{r}\left( x\right) $%
, for $x\in \Omega $, and $0<r<dist\left( x,\partial \Omega \right) $ we
recieve 
\begin{equation*}
\frac{\rho +1}{\left\vert B_{r}\left( x\right) \right\vert }\underset{%
B_{r}\left( x\right) }{\int }\left\vert u\right\vert ^{\rho }\left\vert
\nabla u\right\vert ^{2}~dy=
\end{equation*}

\begin{equation*}
\frac{1}{\left\vert B_{r}\left( x\right) \right\vert }\underset{B_{r}\left(
x\right) }{\int }\nabla \cdot \left( ~\left\vert u\right\vert ^{\rho
}u~\nabla u\right) dy-\frac{1}{\left\vert B_{r}\left( x\right) \right\vert }%
\underset{B_{r}\left( x\right) }{\int }\left\vert u\right\vert ^{\rho
}u\Delta udy
\end{equation*}

or%
\begin{equation*}
\frac{1}{\left\vert B_{r}\left( x\right) \right\vert }\underset{B_{r}\left(
x\right) }{\int }~\left\vert u\right\vert ^{\rho }~\left\vert \nabla
u\right\vert ^{2}dy=
\end{equation*}

\begin{equation}
\frac{1}{\left( \rho +1\right) }\left\{ -\frac{1}{\left\vert B_{r}\left(
x\right) \right\vert }\underset{B_{r}\left( x\right) }{\int }\left\vert
u\right\vert ^{\rho }u\Delta udy+\frac{1}{\left\vert B_{r}\left( x\right)
\right\vert }\underset{S_{r}\left( x\right) }{\int }\left\vert u\right\vert
^{\rho }u\frac{\partial u}{\partial \nu }~dS\left( y\right) \right\}
\label{Eqn 3.19}
\end{equation}

Using (\ref{Eqn 3.19}), the left part of (\ref{Eqn 3.15}) is estimated in
the following way 
\begin{equation*}
\underset{\Omega }{\int }\left[ \left\vert u\right\vert ^{\rho }\left\vert
\nabla u\right\vert ^{2}\right] ^{\beta }dx\leq c\underset{\Omega }{\int }%
\left\vert \left\vert u\right\vert ^{\rho }\left\vert \nabla u\right\vert
^{2}-\frac{1}{\left\vert B_{r}\left( x\right) \right\vert }\underset{%
B_{r}\left( x\right) }{\int }\left\vert u\right\vert ^{\rho }\left\vert
\nabla u\right\vert ^{2}~dy\right\vert ^{\beta }dx+
\end{equation*}%
\begin{equation}
c\underset{\Omega }{\int }\left\vert \frac{1}{\left\vert B_{r}\left(
x\right) \right\vert }\underset{B_{r}\left( x\right) }{\int }\left\vert
u\right\vert ^{\rho }\left\vert \nabla u\right\vert ^{2}~dy\right\vert
^{\beta }dx=I_{1}(r)+I_{2}(r)  \label{Eqn 3.20}
\end{equation}%
According to (\ref{Eqn 3.16}) we have $\underset{r\rightarrow 0}{\lim }$ $%
I_{1}(r)=0$. Therefore, it is enough to show that $I_{2}(r)$ is estimated
uniformly with respect to the $r.$Taking (\ref{Eqn 3.17}), (\ref{Eqn 3.18})
and (\ref{Eqn 3.19}) into consideration in $I_{2}(r)$ we get 
\begin{equation*}
I_{2}(r)=c\underset{\Omega }{\int }\left\vert \frac{1}{\left\vert
B_{r}\left( x\right) \right\vert }\underset{B_{r}\left( x\right) }{\int }%
\left\vert u\right\vert ^{\rho }\left\vert \nabla u\right\vert
^{2}~dy\right\vert ^{\beta }dx=
\end{equation*}%
\begin{equation*}
c\underset{\Omega }{\int }\left\vert -\frac{1}{\left\vert B_{r}\left(
x\right) \right\vert }\underset{B_{r}\left( x\right) }{\int }\left\vert
u\right\vert ^{\rho }u\Delta udy+\frac{1}{\left\vert B_{r}\left( x\right)
\right\vert }\underset{S_{r}\left( x\right) }{\int }\left\vert u\right\vert
^{\rho }u\frac{\partial u}{\partial \nu }~dS\left( y\right) \right\vert
^{\beta }dx\leq
\end{equation*}%
\begin{equation*}
c_{1}\underset{\Omega }{\int }\left\vert \frac{1}{\left\vert B_{r}\left(
x\right) \right\vert }\underset{B_{r}\left( x\right) }{\int }\left\vert
u\right\vert ^{\rho }u\Delta udy\right\vert ^{\beta }dx+c_{1}\underset{%
\Omega }{\int }\left\vert \frac{1}{\left\vert B_{r}\left( x\right)
\right\vert }\underset{S_{r}\left( x\right) }{\int }\left\vert u\right\vert
^{\rho }u\frac{\partial u}{\partial \nu }~dS\left( y\right) \right\vert
^{\beta }dx\leq
\end{equation*}%
\begin{equation*}
c_{1}\underset{\Omega }{\int }\left\vert \left\vert u\right\vert ^{\rho
}u\Delta u\right\vert ^{\beta }dx+c_{1}\underset{\Omega }{\int }\left\vert
\left\vert u\right\vert ^{\rho }u\nabla u\right\vert ^{\beta }dx\leq
\end{equation*}%
\begin{equation*}
c_{1}\underset{\Omega }{\int }\left\vert \left\vert u\right\vert ^{\rho
}u\Delta u\right\vert ^{\beta }dx+\varepsilon \underset{\Omega }{\int }%
\left\vert \left\vert u\right\vert ^{\rho }\left\vert \nabla u\right\vert
^{2}\right\vert ^{\beta }dx+c_{2}\left( \varepsilon \right) \underset{\Omega 
}{\int }\left\vert u\right\vert ^{\left( \rho +2\right) \beta }dx.
\end{equation*}%
Consequently 
\begin{equation}
I_{2}(r)\leq c_{1}\underset{\Omega }{\int }\left\vert u\right\vert ^{\alpha
+\beta }\left\vert \Delta u\right\vert ^{\beta }dx+\varepsilon \underset{%
\Omega }{\int }\left\vert u\right\vert ^{\alpha }\left\vert \nabla
u\right\vert ^{2\beta }dx+c_{2}\left( \varepsilon \right) \underset{\Omega }{%
\int }\left\vert u\right\vert ^{\alpha +2\beta }dx.  \label{Eqn 3.21}
\end{equation}

where $c_{1}$,$c_{2}$ are positive quantities not dependend on $r.$ Choosing
sufficiently small $\varepsilon >0$, such that $\varepsilon <1$, then
substituting the right side of (\ref{Eqn 3.21}) into (\ref{Eqn 3.20}) and
passing to the limit by $r\searrow 0$ in the obtained inequality we get the
desired inequality (\ref{Eqn 3.15}).
\end{proof}

Proposition \ref{Proposition 3.2} and Lemma \ref{Lemma 3.1} imply

\begin{corollary}
\label{Corollary 3.1}Under the conditions of Lemma \ref{Lemma 3.1} the
inequality 
\begin{equation}
\underset{\Omega }{\int }\left\vert u\right\vert ^{\alpha }\left\vert \nabla
u\right\vert ^{2\beta }dx\leq c\underset{\Omega }{\int }\left\vert
u\right\vert ^{\alpha +\beta }\left\vert \Delta u\right\vert ^{\beta }dx
\label{Eqn 3.22}
\end{equation}%
holds with $c=c\left( \alpha ,\beta \right) $ that is not dependent on $u$.
\end{corollary}

\bigskip Our next goal is considering relations between the spaces $W_{\beta
}^{2}\left( \Omega \right) $ and .$S_{\Delta ,\alpha ,\beta }\left( \Omega
\right) .$

We start with definition of the second order Sobolev space: 
\begin{equation*}
W_{\beta }^{2}\left( \Omega \right) \equiv \left\{ u\in L_{1}\left( \Omega
\right) ~\left\vert \ u,D_{i}u,D_{i}D_{j}u\in L_{\beta }\left( \Omega
\right) ,\right. \ i,j=\overline{1,n}\right\} .
\end{equation*}%
\qquad \qquad

It is well known (\cite{BIN}) that 
\begin{equation*}
W_{\beta }^{2}\left( \Omega \right) \equiv \left\{ u\in L_{1}\left( \Omega
\right) ~\left\vert u,D_{i}^{2}u\in L_{\beta }\left( \Omega \right) ,\ i=%
\overline{1,n}\right. \right\} .
\end{equation*}

Moreover, for sufficiently smooth domains (\cite{ADN, BIN}).

\begin{equation*}
W_{\beta }^{2}\left( \Omega \right) \cap \overset{0}{W}\/_{\beta }^{1}\left(
\Omega \right) \equiv \left\{ u~\left\vert ~\Delta u\in L_{\beta }\left(
\Omega \right) ,\ u\left\vert _{\partial \Omega }=0\right. \right. \right\}
\end{equation*}

We also define the following class of functions 
\begin{equation}
\mathcal{M}_{\Delta \circ \eta ,L_{\beta }\left( \Omega \right) }\equiv
\left\{ u~\left\vert ~\Delta \circ ~\eta \left( u\right) \in L_{\beta
}\left( \Omega \right) ,\ \eta \left( u\right) \equiv \left\vert
u\right\vert ^{\frac{\alpha }{\beta }}u\right. \right\} .  \label{Eqn 3.23}
\end{equation}

Now, we are ready to compare spaces defined in (\ref{Eqn 3.5}) and (\ref{Eqn
3.23})

\begin{lemma}
\label{Lemma 3.2}Let $\alpha \geq 0$, $\alpha _{1}>-1$, $\beta _{1}\geq
\beta \geq \frac{\beta _{1}}{2}\geq 1$, $\beta >\frac{n}{n-1}$ be some
numbers, $\alpha +\beta =\alpha _{1}+\beta _{1}$, (if $\beta _{1}=2\beta $
then $\alpha >\beta -1$) and $\Omega \subset \Re ^{n},$ $n\geq 1,$ be a
bounded domain with sufficiently smooth boundary $\partial \Omega $. Then,
the following inclusion 
\begin{equation*}
\overset{0}{S}_{\Delta ,\alpha ,\beta }\left( \Omega \right) =S_{\Delta
,\alpha ,\beta }\left( \Omega \right) \cap \left\{ u\left\vert ~u\left\vert
_{\partial \Omega }\right. =0\right. \right\} \subseteq \mathcal{M}_{\Delta
\circ \eta ,L_{\beta }\left( \Omega \right) }\cap \left\{ u\left\vert
~u\left\vert _{\partial \Omega }\right. =0\right. \right\}
\end{equation*}%
takes place.
\end{lemma}

\begin{proof}
Let $u\in \overset{0}{S}_{\Delta ,\alpha ,\beta }\left( \Omega \right) $ be
an arbitrary function. Then, according to (\ref{Eqn 3.23}) $\eta _{1}\left(
u\right) \equiv \left\vert u\left( x\right) \right\vert ^{\frac{\alpha _{1}}{%
\beta _{1}}}u\left( x\right) \in \overset{0}{W}$\/$_{\beta _{1}}^{1}\left(
\Omega \right) $ as far as $u\left\vert _{\partial \Omega }\right. =0$ $%
\Longleftrightarrow \eta _{1}\left( u\right) \left\vert _{\partial \Omega
}\right. =0,$ and $\left\vert u\right\vert ^{\frac{\alpha }{\beta }}~\Delta
u\in L_{\beta }\left( \Omega \right) $. Moreover, if $\alpha _{1}>0$ then $%
u\left\vert _{\partial \Omega }\right. =0$ $\Longrightarrow $ $\frac{%
\partial }{\partial n}\eta _{1}\left( u\right) \left\vert _{\partial \Omega
}\right. =0$ .

Under the conditions of Lemma, according to (\ref{Eqn 3.22}) and (\ref{Eqn
3.11}), the inequality 
\begin{equation*}
\underset{\Omega }{\int }\left\vert u\right\vert ^{\alpha _{1}}\left\vert
\nabla u\right\vert ^{\beta _{1}}dx\leq c\underset{\Omega }{\int }\left\vert
u\right\vert ^{\alpha }\left\vert \Delta u\right\vert ^{\beta }dx
\end{equation*}%
takes place with $c=c\left( \alpha ,\beta ,\alpha _{1},\beta _{1}\right) $
that is not dependent on $u.$

Taking this and Corollary \ref{Corollary 3.1} into account we conclude 
\begin{equation}
\overset{0}{S}_{\Delta ,\alpha ,\beta }\left( \Omega \right) \equiv \left\{
u\left( x\right) \left\vert \ \left\vert u\right\vert ^{\frac{\alpha }{\beta 
}}~\Delta u\in L_{\beta }\left( \Omega \right) \right. \right\} \cap \left\{
u\left\vert ~u\left\vert _{\partial \Omega }\right. =0\right. \right\} ,
\label{Eqn 3.24}
\end{equation}

On the other hand, the definition (\ref{Eqn 3.23}) implies that $u\in 
\mathcal{M}_{\Delta \circ \eta ,L_{\beta }\left( \Omega \right) }$ is
equivalent to $\Delta v\equiv \Delta \eta \left( u\right) \in L_{\beta
}\left( \Omega \right) $. Indeed, using (\ref{Eqn 3.6}) and estimating $%
L_{\beta }\left( \Omega \right) $ of $\Delta v$ we get 
\begin{equation*}
\underset{\Omega }{\left\Vert \Delta v\right\Vert _{L_{\beta }\left( \Omega
\right) }^{\beta }=\int }\left\vert \left\vert u\right\vert ^{\frac{\alpha }{%
\beta }}~\Delta u+\frac{\alpha }{\beta }~\left\vert u\right\vert ^{\frac{%
\alpha }{\beta }-2}~u~\left\vert \nabla u\right\vert ^{2}\right\vert ^{\beta
}dx\leq
\end{equation*}%
\begin{equation*}
c\left\{ \underset{\Omega }{\int }\left\vert \left\vert u\right\vert ^{\frac{%
\alpha }{\beta }}~\Delta u\right\vert ^{\beta }dx+\underset{\Omega }{\int }%
\left\vert ~\left\vert u\right\vert ^{\frac{\alpha }{\beta }-1}~\left\vert
\nabla u\right\vert ^{2}\right\vert ^{\beta }dx\right\} .
\end{equation*}%
Taking the inequality (\ref{Eqn 3.22}) and the equivalence (\ref{Eqn 3.24})
into account we obtain 
\begin{equation*}
\overset{0}{S}_{\Delta ,\alpha ,\beta }\left( \Omega \right) \subseteq 
\mathcal{M}_{\Delta \circ \eta ,L_{\beta }\left( \Omega \right) }\cap
\left\{ u\left\vert ~u\left\vert _{\partial \Omega }\right. =0\right.
\right\}
\end{equation*}
\end{proof}

\begin{corollary}
\label{Corollary 3.2}Under the conditions of Lemma \ref{Lemma 3.2} the
implication 
\begin{equation*}
u\in \overset{0}{S}_{\Delta ,\alpha ,\beta }\left( \Omega \right)
\Longrightarrow v\equiv \eta \left( u\right) \in W_{\beta }^{2}\left( \Omega
\right) \cap \overset{0}{W}~_{\beta }^{1}\left( \Omega \right)
\end{equation*}%
holds.\ \ 
\end{corollary}

\begin{proof}
If $v\left( x\right) \equiv \eta \left( u\right) \equiv \left\vert u\left(
x\right) \right\vert ^{\frac{\alpha }{\beta }}u\left( x\right) $ then $%
\nabla v\left( x\right) \equiv \left( \frac{\alpha }{\beta }+1\right)
\left\vert u\left( x\right) \right\vert ^{\frac{\alpha }{\beta }}\nabla
u\left( x\right) $ and (\ref{Eqn 3.6}) takes place for $\Delta v$.

According to mentioned above, the inclusion $v\equiv \eta \left( u\right)
\in W_{\beta }^{2}\left( \Omega \right) \cap \overset{0}{W}$\/$_{\beta
}^{1}\left( \Omega \right) $ is equivalent to $\Delta v=\Delta \circ \eta
\left( u\right) \in L_{\beta }\left( \Omega \right) $, as far as $%
u\left\vert _{\partial \Omega }\right. =0$ $\iff \eta \left( u\right)
\left\vert _{\partial \Omega }\right. =0$ (moreover, if $\alpha >0$ then $%
u\left\vert _{\partial \Omega }\right. =0$ $\Longrightarrow $ $\frac{%
\partial }{\partial n}\eta \left( u\right) \left\vert _{\partial \Omega
}\right. =0$). This implies that $\eta \left( u\right) \in W_{\beta
}^{2}\left( \Omega \right) \cap \overset{0}{W}$\/$_{\beta }^{1}\left( \Omega
\right) $ is equ\i valent to $u\in M_{\Delta \circ \eta ,L_{\beta }\left(
\Omega \right) }\cap \left\{ u\left\vert ~u\left\vert _{\partial \Omega
}\right. =0\right. \right\} $ Therefore, taking the lemma into account we
conclude the desired implication.
\end{proof}

\begin{notation}
\label{Notation 3.1}If parameters $\alpha ,\alpha _{1}\geq 0$, $\beta ,\beta
_{1},p,p_{0},p_{1}\geq 1$ satisfy certain conditions some relations between
spaces $S_{\Delta ,\alpha ,\beta }\left( \Omega \right) $, $S_{1,\alpha
_{1},\beta _{1}}\left( \Omega \right) $, $L_{p}\left( \Omega \right) $, $%
P_{p_{0},p_{1}}\left( 0,T;X;S_{\Delta ,\alpha ,\beta }\left( \Omega \right)
\right) $, $L_{p}\left( 0,T;S_{1,\alpha _{1},\beta _{1}}\left( \Omega
\right) \right) $ can be obtained according to their definitions. More
precise inclusion and compactness results for them can be proved on the way
that is similar to our earlier works \cite{S2, S3, S4, S5}. Here, we are
presenting some of such type of results.
\end{notation}

\begin{theorem}
\label{Theorem 3.1}Let $\alpha ,\alpha _{1}\geq 0$, $\beta >\frac{n}{n-1}%
,\beta _{1}\geq 1$ be such numbers that $\frac{\alpha _{1}+\beta _{1}}{%
\alpha +\beta }\geq \beta \beta _{1}^{-1}$ and $\alpha \beta _{1}\geq \alpha
_{1}\beta $, $\alpha >\beta -1$ Then $\overset{0}{S}_{\Delta ,\alpha ,\beta
}\left( \Omega \right) $ $\subset $ $S_{1,\alpha _{1},\beta _{1}}\left(
\Omega \right) $.
\end{theorem}

The proof follows from the inequality 
\begin{equation*}
\underset{\Omega }{\int }\left\vert u\right\vert ^{\alpha _{1}}\left\vert
\nabla u\right\vert ^{\beta _{1}}dx\leq c\left( \varepsilon \right) \underset%
{\Omega }{\int }\left\vert u\right\vert ^{\alpha }\left\vert \Delta
u\right\vert ^{\beta }dx+\varepsilon \left( \underset{\Omega }{\int }%
\left\vert u\right\vert ^{s}dx\right) ^{\frac{\alpha +\beta }{s}},\ 
\end{equation*}%
where $\ s=s\left( \alpha ,\alpha _{1},\beta ,\beta _{1}\right) \leq \alpha
+\beta $, that can be derived by using the inequalities (\ref{Eqn 3.10}), (%
\ref{Eqn 3.11}), (\ref{Eqn 3.15}) and (\ref{Eqn 3.22}) (for details refer to 
\cite{S3, S4}).

\begin{remark}
\label{Remark 3.2}Note that it is not difficult to verify that if $\frac{%
\alpha +\beta }{\alpha _{1}+\beta _{1}}\geq 1$, $\beta \geq \beta _{1}$ and $%
\frac{n\left( \alpha _{1}+\beta _{1}\right) }{n-\beta _{1}}\geq p$, $n>\beta
_{1}$ then the following inclusions 
\begin{equation*}
S_{1,\alpha ,\beta }\left( \Omega \right) \subseteq S_{1,\alpha _{1},\beta
_{1}}\left( \Omega \right) \subset L_{p}\left( \Omega \right) ,\text{ \ \ }%
S_{\Delta ,\alpha ,\beta }\left( \Omega \right) \subseteq S_{\Delta ,\alpha
_{1},\beta _{1}}\left( \Omega \right)
\end{equation*}%
take place. \ Moreover, arguments similar to those that express relations
between the considered and Sobolev spaces show that the inclusion $%
S_{1,\alpha ,\beta }\left( \Omega \right) \subset L_{p}\left( \Omega \right) 
$ and consequently, $\overset{0}{S}_{\Delta ,\alpha ,\beta }\left( \Omega
\right) \subset L_{p}\left( \Omega \right) $ are compact (for detail one can
refer to \cite{S3, S4, S5}).
\end{remark}

\begin{corollary}
\label{Corollary 3.3}Assume that the conditions of Theorem \ref{Theorem 3.1}
are fulfilled. Then, the following inclusions 
\begin{equation*}
P_{p_{0},p_{1}}\left( 0,T;S_{\Delta ,\alpha ,\beta }\left( \Omega \right)
;X\right) \subseteq P_{\widetilde{p}_{0},\widetilde{p}_{1}}\left(
0,T;S_{\Delta ,\alpha _{1},\beta _{1}}\left( \Omega \right) ;\widetilde{X}%
\right) ,
\end{equation*}%
\begin{equation*}
P_{p_{0},p_{1}}\left( 0,T;S_{\Delta ,\alpha ,\beta }\left( \Omega \right)
;X\right) \subset L_{p}\left( 0,T;S_{1,\alpha _{1},\beta _{1}}\left( \Omega
\right) \right) ,
\end{equation*}%
hold if $X\subseteq \widetilde{X}$, and $p_{0}\geq \widetilde{p}_{0}\geq
1,p_{1}\geq \widetilde{p}_{1}\geq 1$, $p_{1}\geq p\geq 1$,
\end{corollary}

\begin{remark}
\label{Remark 3.3}If $\alpha \geq 0$, $\frac{\alpha _{1}}{\beta _{1}}>-1$, $%
\frac{1}{2}\beta _{1}=\beta >\frac{n}{n-1}$ such numbers that $\alpha +\beta
=\alpha _{1}+\beta _{1}$, $\alpha >\beta -1$, then 
\begin{equation*}
S_{\Delta ,\alpha ,\beta }\left( \Omega \right) \iff \left\{ u\left(
x\right) \left\vert ~\eta \left( u\right) \equiv \left\vert u\right\vert
^{\rho }u\in W_{\beta }^{2}\left( \Omega \right) \right. \ \rho =\frac{%
\alpha }{\beta }\right\} ,
\end{equation*}%
i.e. 
\begin{equation*}
u\in S_{\Delta ,\alpha ,\beta }\left( \Omega \right) \Longrightarrow v\equiv
\eta \left( u\right) \equiv \left\vert u\right\vert ^{\rho }u\in W_{\beta
}^{2}\left( \Omega \right) \Longrightarrow
\end{equation*}%
\begin{equation*}
u\equiv \eta ^{-1}\left( v\right) \equiv \left\vert v\right\vert ^{-\frac{%
\rho }{\rho +1}}v\in S_{\Delta ,\alpha ,\beta }\left( \Omega \right)
\end{equation*}%
under the conditions (see, \cite{M} and also \cite{S3, S4}) that all
operations make a seinse.

Furthermore, note that $S_{1,\alpha ,\beta }\left( \Omega \right) $ and $%
S_{\Delta ,\alpha ,\beta }\left( \Omega \right) $ are metric spaces (\cite%
{S3, S4, S5}) with the corresponding metrics of the form: 
\begin{equation*}
d_{S_{1,\alpha ,\beta }\left( \Omega \right) }\left( u;v\right) \equiv
\left\Vert \eta \left( u\right) -\eta \left( v\right) \right\Vert _{W_{\beta
}^{1}\left( \Omega \right) }^{\left( \rho +1\right) ^{-1}},\quad \eta \left(
\tau \right) \equiv \left\vert \tau \right\vert ^{\rho }\tau ,\rho =\frac{%
\alpha }{\beta },\alpha \geq 0,\beta >1,
\end{equation*}%
\begin{equation*}
d_{S_{\Delta ,\alpha ,\beta }\left( \Omega \right) }\left( u;v\right) \equiv
\left\Vert \eta _{1}\left( u\right) -\eta _{1}\left( v\right) \right\Vert
_{W_{\beta _{1}}^{1}\left( \Omega \right) }^{\left( \rho _{1}+1\right)
^{-1}}+\left\Vert \left\vert u\right\vert ^{\rho }\Delta u-\left\vert
v\right\vert ^{\rho }\Delta v\right\Vert _{L_{\beta }\left( \Omega \right)
}^{\left( \rho +1\right) ^{-1}},
\end{equation*}%
where $\rho _{1}=\frac{\alpha _{1}}{\beta _{1}}$, $\eta _{1}\left( \tau
\right) \equiv \left\vert \tau \right\vert ^{\rho _{1}}\tau $, and $\alpha
_{1}+\beta _{1}=\alpha +\beta $ (see (\ref{Eqn 2.1})).

Moreover, it is not difficult to see that the metrics of spaces $\overset{0}{%
S}_{1,\alpha ,\beta }\left( \Omega \right) $ and $\overset{0}{S}_{\Delta
,\alpha ,\beta }\left( \Omega \right) $ have the form: 
\begin{equation*}
d_{\overset{0}{S}_{1,\alpha ,\beta }\left( \Omega \right) }\left( u;v\right)
\equiv \left\Vert \left\vert \left\vert u\right\vert ^{\rho }\nabla
u-\left\vert v\right\vert ^{\rho }\nabla v\right\vert \right\Vert _{L_{\beta
}\left( \Omega \right) }^{\left( \rho +1\right) ^{-1}};
\end{equation*}%
\begin{equation*}
d_{\overset{0}{S}_{\Delta ,\alpha ,\beta }\left( \Omega \right) }\left(
u;v\right) \equiv \left\Vert \left\vert u\right\vert ^{\rho }\Delta
u-\left\vert v\right\vert ^{\rho }\Delta v\right\Vert _{L_{\beta }\left(
\Omega \right) }^{\left( \rho +1\right) ^{-1}}
\end{equation*}%
correspondently.
\end{remark}

Based on Theorem \ref{Theorem 3.1}, Corollary \ref{Corollary 3.3} and
Embedding Theorems for the Sobolev spaces we prove the following:

\begin{theorem}
\label{Theorem 3.2}Let $\alpha ,\alpha _{1}\geq 0$, $\beta >\frac{n}{n-1}$, $%
\beta _{1}\geq 1$ be such numbers that $\beta _{1}<\frac{n\beta }{n-\beta }$%
, $\beta <n$\ , $\alpha _{1}+\beta _{1}<\frac{n\left( \alpha +\beta \right) 
}{n-\beta }$ and $\alpha \beta _{1}\geq \alpha _{1}\beta $, $\alpha >\beta
-1 $. Then, the inclusion 
\begin{equation*}
\overset{0}{S}_{\Delta ,\alpha ,\beta }\left( \Omega \right) \subset \overset%
{0}{S}_{1,\alpha _{1},\beta _{1}}\left( \Omega \right)
\end{equation*}

is compact.
\end{theorem}

\begin{proof}
Since $u\in \overset{0}{S}_{\Delta ,\alpha ,\beta }\left( \Omega \right) ,$
we have $\eta \left( u\right) \equiv v\in W_{\beta }^{2}\left( \Omega
\right) \cap \overset{0}{W}$\/$_{\beta }^{1}\left( \Omega \right) $ and $%
u\in \overset{0}{S}_{1,\alpha _{1},\beta _{1}}\left( \Omega \right)
\Longleftarrow \Longrightarrow \eta _{1}\left( u\right) \equiv v\in \overset{%
0}{W}$\/$_{\beta _{1}}^{1}\left( \Omega \right) $ ($\eta _{1}\left( u\right)
\equiv \left\vert u\right\vert ^{\rho _{1}}u,\ \rho _{1}=\frac{\alpha _{1}}{%
\beta _{1}}$). Moreover, as far as $W_{\beta }^{2}\left( \Omega \right)
\subset W~_{\beta _{1}}^{1}\left( \Omega \right) $ is compact for $\beta
_{1}<\frac{n\beta }{n-\beta }$, we get the compactness of the inclusion $%
\eta \left( G\right) \subset \overset{0}{W}$\/$_{\beta _{1}}^{1}\left(
\Omega \right) $ for any bounded subset $G$ from $\overset{0}{S}_{\Delta
,\alpha ,\beta }\left( \Omega \right) $. This implies the desired statement
\end{proof}

\begin{corollary}
\label{Corollary 3.4}If $0<\rho \leq 2,$ then $\overset{0}{S}_{\Delta ,\rho
,2}\left( \Omega \right) \subset \overset{0}{W}$\/$_{p}^{1}\left( \Omega
\right) $, $p=\rho +2$.
\end{corollary}

Mentioned above and known results (\cite{L, LM, S3, S4}) allow us to prove
the the compact embeddings for the following vector spaces: $L_{p}\left(
0,T;S_{1,\alpha _{1},\beta _{1}}\left( \Omega \right) \right) $, $%
P_{1,p_{0},p_{1}}\left( 0,T;S_{\Delta ,\alpha ,\beta }\left( \Omega \right)
;X\right) $. We need the following

\begin{lemma}
\label{Lemma 3.3}Let $\alpha ,\alpha _{1},\alpha _{2}\geq 0$, $\beta ,\beta
_{1},\beta _{2}\geq 1$, $2\beta \geq \widetilde{\beta }\geq 1$ be such
numbers that $\alpha +\beta =\alpha _{1}+\beta _{1}=\alpha _{2}+\beta _{2}$, 
$\beta _{1}<\beta <\beta _{2}$, $1\leq \widetilde{\beta _{1}}<\widetilde{%
\beta }<\widetilde{\beta _{2}}$. Then, for any $\varepsilon >0$ there exists 
$c\left( \varepsilon \right) >0$ such that the inequality 
\begin{equation*}
\left[ u\right] _{S_{\Delta ,\alpha ,\beta }}\leq \varepsilon \left[ u\right]
_{S_{\Delta ,\alpha _{2},\beta _{2}}}+c\left( \varepsilon \right) \left[ u%
\right] _{S_{\Delta ,\alpha _{1},\beta _{1}}},\ \forall u\in S_{\Delta
,\alpha _{2},\beta _{2}}\left( \Omega \right)
\end{equation*}%
holds.
\end{lemma}

The proof is obvious.

\begin{lemma}
\label{Lemma 3.4}Let $\alpha ,\alpha _{1}\geq 0$, $\beta ,\beta _{0},p\geq 1$%
, $\frac{\beta _{1}}{2}\geq \beta >\frac{n}{n-1}$ be such numbers that $%
\beta _{0}<\frac{n\beta }{n-\beta }$, $\beta <n$\ , $\alpha _{0}+\beta _{0}<%
\frac{n\left( \alpha +\beta \right) }{n-\beta }$ and $\alpha \beta _{0}\geq
\alpha _{0}\beta $, $\alpha >\beta -1$, $p\leq \alpha +\beta $. Then, for
any $\varepsilon >0$ there exists $c\left( \varepsilon \right) >0$ such that
the inequality 
\begin{equation*}
d_{S_{1,\alpha _{0},\beta _{0}}}\left( u;v\right) \leq \varepsilon \left( 
\left[ u\right] _{S_{\Delta ,\alpha ,\beta }}+\left[ v\right] _{S_{\Delta
,\alpha ,\beta }}\right) +c\left( \varepsilon \right) \left\Vert
u-v\right\Vert _{L_{p}},\ \forall u,v\in S_{\Delta ,\alpha _{2},\beta
_{2}}\left( \Omega \right)
\end{equation*}%
holds
\end{lemma}

The proof is similar to the proof of the same type results from \cite{D, L,
S3, S4} and is based on the compactness of the inclusion $S_{\Delta ,\alpha
_{2},\beta _{2}}\left( \Omega \right) \subset S_{1,\alpha _{1},\beta
_{1}}\left( \Omega \right) \subset L_{p}\left( \Omega \right) $.

These lemmas allow us to get the following compactness

\begin{theorem}
\label{Theorem 3.3}Let $S_{1,\alpha _{1},\beta _{1}}\left( \Omega \right) $
, $S_{\Delta ,\alpha ,\beta }\left( \Omega \right) $ and $X$ be spaces
defined above and $S_{\Delta ,\alpha ,\beta }\left( \Omega \right) \subset
S_{1,\alpha _{1},\beta _{1}}\left( \Omega \right) $ is compact. Let $\alpha
_{1}\geq 0$, $\beta ,\beta _{1},p,p_{0},p_{1}\geq 1$ be such numbers that $%
\alpha +\beta =p=p_{1}$, $\alpha \beta _{1}\geq \alpha _{1}\beta $, $\beta >%
\frac{n}{n-1}$, $\alpha >\beta -1$. Then, the inclusion $P_{1,p_{0},p_{1}}%
\left( 0,T;S_{\Delta ,\alpha ,\beta }\left( \Omega \right) ;X\right) \subset
L_{p}\left( 0,T;S_{1,\alpha _{1},\beta _{1}}\left( \Omega \right) \right) $
is compact.
\end{theorem}

The proof is similar to the proof of the same type of results from \cite{L,
S2, S3, S4, S5}. Therefore, we are not providing it here. The other
compactness theorems similar to Theorem \ref{Theorem 3.2} and Lemma \ref%
{Lemma 3.3} can also be proved , but we are not presenting them here, as
well. However, if it would be neccessary, we are going to use those theorems
for the spaces $P_{1,p_{0},p_{1}}\left( 0,T;\overset{0}{S}_{\Delta ,\alpha
,\beta }\left( \Omega \right) ;X\right) $, $L_{p}\left( 0,T;\overset{0}{S}%
_{1,\alpha _{1},\beta _{1}}\left( \Omega \right) \right) $ under the
corresponding conditions on parameters $\alpha ,\alpha _{1}$, $\beta ,\beta
_{1}$, $p,p_{0},p_{1}$ and refer reader to our earlier works \cite{S2, S3,
S4, S5} for further details.

\section{\textbf{The Proof of the Solvability Theorem}}

Now we can lead the proof by using Theorem \ref{Theorem 2.2} (Corollary \ref%
{Corollary 2.1}), and in order to apply it we introduce the following spaces
and mappings:

\begin{equation*}
\mathcal{M}_{0}\equiv \overset{0}{S}_{\Delta ,\rho ,2}\left( \Omega \right)
,\quad X_{0}\equiv W_{p}^{2}\left( \Omega \right) \cap \overset{0}{W}\text{\/%
}_{p}^{1}\left( \Omega \right) ,\qquad Y\equiv L_{q}\left( \Omega \right) ,\
X\equiv L_{p}\left( \Omega \right) ,
\end{equation*}%
\begin{equation*}
f\left( u\right) \equiv -\left\vert u\right\vert ^{\rho }\Delta
u+b_{0}\left\vert u\right\vert ^{\mu +1},\quad L\equiv -\Delta ,\quad
L_{0}\equiv \nabla ,\qquad Y^{\ast }\equiv L_{p}\left( Q\right) ,\ p=\rho +2
\end{equation*}%
\begin{equation*}
\underset{0}{\mathbf{P}}\/_{1,p,q}\left( 0,T;\mathcal{M}_{0},Y\right) \equiv 
\underset{0}{\mathbf{P}}\/_{1,p,q}\left( 0,T;\overset{0}{S}_{\Delta ,\rho
,2}\left( \Omega \right) ;L_{q}\left( \Omega \right) \right) \cap L^{\infty
}\left( 0,T;\overset{0}{W}\text{\/}_{2}^{1}\left( \Omega \right) \right)
\end{equation*}%
where

\begin{equation*}
\underset{0}{\mathbf{P}}\text{\/}_{1,p,q}\left( 0,T;\overset{0}{S}_{\Delta
,\rho ,2}\left( \Omega \right) ;L_{q}\left( \Omega \right) \right) \equiv
L_{p}\left( 0,T;\overset{0}{S}_{\Delta ,\rho ,2}\left( \Omega \right)
\right) \cap \underset{0}{W}\text{\/}_{q}^{1}\left( 0,T;L_{q}\left( \Omega
\right) \right) .
\end{equation*}

It not is difficult to see that 
\begin{equation*}
\left\langle f\left( u\right) ,Lu\right\rangle \equiv \left\langle
-\left\vert u\right\vert ^{\rho }\Delta u+b_{0}\left\vert u\right\vert ^{\mu
+1},-\Delta u\right\rangle =\underset{\Omega }{\int }~\left\vert
u\right\vert ^{\rho }\left( \Delta u\right) ^{2}dx+\underset{\Omega }{\int }%
~b_{0}\left\vert u\right\vert ^{\mu +1}\Delta udx
\end{equation*}
for any $u\in W_{p}^{2}\left( \Omega \right) \cap \overset{0}{W}$\/$%
_{p}^{1}\left( \Omega \right) $ and $u\in L_{p}\left( 0,T;W_{p}^{2}\left(
\Omega \right) \cap \overset{0}{W}\text{\/}_{p}^{1}\left( \Omega \right)
\right) $.

Taking into account the embedding theorems from Section 3, the last equality
implies that, if $\min \left\{ 0,\frac{\rho }{2}-1\right\} \leq \mu <\rho
\leq 2$ or\ $\frac{\rho }{2}-1\leq \mu <\rho $ and $b_{0}\in R^{1}$ then 
\begin{equation*}
\underset{0}{\overset{T}{\int }}\left\langle f\left( u\right)
,Lu\right\rangle dt\geq \left( 1-\varepsilon \right) \underset{0}{\overset{T}%
{\int }}\ \underset{\Omega }{\int }~\left\vert u\right\vert ^{\rho }\left(
\Delta u\right) ^{2}dxdt-c_{1}\left( \varepsilon \right) =
\end{equation*}%
\begin{equation*}
\left( 1-\varepsilon \right) \underset{0}{\overset{T}{\int }}\left[ u\right]
_{\overset{0}{S}_{\Delta ,\rho ,2}\left( \Omega \right) }^{\rho
+2}dt-c_{1}\left( \varepsilon \right) \equiv \varphi \left( \left[ u\right]
_{L_{\rho +2}\left( \overset{0}{S}_{\Delta ,\rho ,2}\right) }\right) \ \left[
u\right] _{L_{\rho +2}\left( \overset{0}{S}_{\Delta ,\rho ,2}\right) },
\end{equation*}%
where $c_{0}>0,\ c_{1},\varepsilon \geq 0$, and $\varepsilon $ is a
sufficiently small nonnegative number.

Furthermore, it is obvious that $\underset{0}{\overset{t}{\int }}%
\left\langle \frac{\partial u}{\partial \tau },Lu\right\rangle d\tau \equiv 
\frac{1}{2}\left\Vert \left\vert \nabla u\left( t\right) \right\vert
\right\Vert _{L_{2}}^{2}$ for any $u\in \underset{0}{W}~_{p}^{1}\left(
0,T;X_{0}\right) $ and almost any $t\in \left( 0,T\right] $. Moreover, $%
\underset{0}{\overset{T}{\int }}\left\langle w,Lw\right\rangle d\tau \equiv
\left\Vert \left\vert \nabla w\right\vert \right\Vert _{L_{2}\left( Q\right)
}^{2}$ for any $w\in L_{p}\left( 0,T;X_{0}\right) $, where $w\equiv \frac{%
\partial u}{\partial \tau }\in L_{p}\left( 0,T;X_{0}\right) $.

Using the generalized coercivity of pair $f$ and $-\Delta $ on $L_{p}\left(
0,T;W_{p}^{2}\left( \Omega \right) \cap \overset{0}{W}\text{\/}%
_{p}^{1}\left( \Omega \right) \right) \cap W_{q}^{1}\left( 0,T;L_{q}\left(
\Omega \right) \right) $ the following apriori estimations for a solution $%
u\left( t,x\right) $ of considered problem are obtained in a common way: 
\begin{equation*}
\left[ u\right] _{L_{p}\left( \overset{0}{S}_{\Delta ,\rho ,2}\right) }\leq
c\left( \left\vert b_{0}\right\vert ,\left\Vert h\right\Vert _{L_{2}\left(
W_{2}^{1}\left( \Omega \right) \right) },\rho ,\mu \right) ,\quad p=\rho
+2,\ q=p^{\prime }
\end{equation*}%
and 
\begin{equation*}
\left\Vert u\right\Vert _{W_{q}^{1}\left( L_{q}\right) \cap L^{\infty
}\left( W_{2}^{1}\right) }\leq c\left( \left\vert b_{0}\right\vert
,\left\Vert h\right\Vert _{L_{2}\left( W_{2}^{1}\left( \Omega \right)
\right) },\rho ,\mu \right) .
\end{equation*}

Thus, each possible solution $u\left( t,x\right) $ of the considered problem
belongs to a bounded subset of 
\begin{equation*}
L_{p}\left( 0,T;\overset{0}{S}_{\Delta ,\rho ,2}\left( \Omega \right)
\right) \cap W_{q}^{1}\left( 0,T;L_{q}\left( \Omega \right) \right) \cap
L^{\infty }\left( 0,T;\overset{0}{W}\text{\/}_{2}^{1}\left( \Omega \right)
\right) ,
\end{equation*}%
and, consequently, the solutions belong to a bounded subset of $%
P_{1,p,q}\left( 0,T;\overset{0}{S}_{\Delta ,\rho ,2}\left( \Omega \right)
;L_{q}\left( \Omega \right) \right) $ and $L^{\infty }\left( 0,T;\overset{0}{%
W}\text{\/}_{2}^{1}\left( \Omega \right) \right) $.

To apply Theorem 2 (Corollary 1) it remains to show that $f$ is a weakly
compact (continuous) mapping from $P_{1}\left( Q_{T}\right) \equiv
P_{1,p,q}\left( 0,T;\overset{0}{S}_{\Delta ,\rho ,2}\left( \Omega \right)
;L_{q}\left( \Omega \right) \right) \cap L^{\infty }\left( 0,T;\overset{0}{W}%
\text{\/}_{2}^{1}\left( \Omega \right) \right) $ into $L_{q}\left(
Q_{T}\right) $. To this end, it is enough to use the following expressions: 
\begin{equation}
\left\vert u\right\vert ^{\rho -2}~u~\left\vert \nabla u\right\vert
^{2}=\left( \left\vert u\right\vert ^{\gamma \rho }~\nabla u\right) \cdot
\left( \left\vert u\right\vert ^{\left( 1-\gamma \right) \rho -2}~u~\nabla
u\right) ,  \label{Eqn 4.1}
\end{equation}%
\begin{equation}
\left\vert u\right\vert ^{\rho -2}~u~\left\vert \nabla u\right\vert ^{2}=%
\frac{1}{\rho \left( \rho +1\right) \left( 1-\theta \right) }~\Delta \left(
\left\vert u\right\vert ^{\rho }u\right) -\frac{1}{\rho \left( \theta \rho
+1\right) \left( 1-\theta \right) }\left\vert u\right\vert ^{\left( 1-\theta
\right) \rho }\ \Delta \left( \left\vert u\right\vert ^{\theta \rho
}u\right) ,  \label{Eqn 4.2}
\end{equation}%
because of 
\begin{equation*}
\left\vert u\right\vert ^{\rho }~\Delta u=\frac{1}{\rho +1}~\Delta \left(
\left\vert u\right\vert ^{\rho }u\right) -\rho \left\vert u\right\vert
^{\rho -2}~u~\left\vert \nabla u\right\vert ^{2},
\end{equation*}%
where $\gamma $ is a number from condition 4) if $\rho \geq 1,$ and $\theta $
is such a number that $\frac{1}{2}\leq \theta <1$ if $0<\rho \leq 2$.
Particularly, if $\theta =\frac{2}{3}$ it is sufficient to use the
expression: 
\begin{equation*}
\left\vert u\right\vert ^{\rho -2}~u~\left\vert \nabla u\right\vert ^{2}=%
\frac{3}{\rho \left( \rho +1\right) }~\Delta \left( \left\vert u\right\vert
^{\rho }u\right) -\frac{3}{\rho }\left\vert u\right\vert ^{\frac{\rho }{3}}\
\nabla \cdot \left( \left\vert u\right\vert ^{\frac{2\rho }{3}}\nabla
u\right)
\end{equation*}%
then $\left\vert u\right\vert ^{\frac{2\rho }{3}}\Delta u\in L_{\frac{%
2\left( \rho +2\right) }{\rho +4}}\left( Q\right) $ and $\left\vert
u\right\vert ^{\frac{\rho }{3}}\in L_{\frac{2\left( \rho +2\right) }{\rho }%
}\left( Q\right) $.

Thus, according to the embedding theorems mentioned above, the solution $%
u\left( t,x\right) \in P_{1,p_{0},q_{0}}\left( 0,T;\overset{0}{S}_{\Delta
,\alpha ,\beta }\left( \Omega \right) ;L_{q}\left( \Omega \right) \right) $
and $u\left( t,x\right) \in L_{p}\left( 0,T;\overset{0}{S}_{1,\alpha
_{1},\beta _{1}}\left( \Omega \right) \right) $ , if the parameters $\alpha
_{1}\geq 0$, $\beta ,\beta _{1},p,p_{0},p_{1}\geq 1$, $\alpha >\beta -1$
satisfy one of the following conditions: 1) $\alpha _{1}=\left( \rho
-1\right) q$, $\beta _{1}=2q$, $p=\rho +2$, $q=p^{\prime }=\frac{\rho +2}{%
\rho +1};$ 2) $\alpha =\rho $, $\beta =2,$ $p_{0}=\rho +2,q_{0}=q$; 3) $%
\alpha =s\beta $, $\beta >1,$ $p_{0}=\left( s+1\right) \beta $, $q_{0}=\beta 
$, $1\leq s\leq \frac{3\rho -2}{4}$; or 4) $\alpha =\gamma \rho \beta $, $%
\beta =\frac{\rho +2}{\gamma \rho +1}$, $\frac{1}{2}\leq \gamma <1$, $%
p_{0}=\rho +2$, $q_{0}=q$.

Therefore, if a sequence $\left\{ u_{m}\right\} \subset P_{1,p,q}\left( 0,T;%
\overset{0}{S}_{\Delta ,\rho ,2}\left( \Omega \right) ;L_{q}\left( \Omega
\right) \right) $ converges weakly to $u\in P_{1,p,q}\left( 0,T;\overset{0}{S%
}_{\Delta ,\rho ,2}\left( \Omega \right) ;L_{q}\left( \Omega \right) \right) 
$ in $P_{1,p,q}\left( 0,T;\overset{0}{S}_{\Delta ,\rho ,2}\left( \Omega
\right) ;L_{q}\left( \Omega \right) \right) $ then, according to the
compactness theorem from Section 3, one of the factors in (\ref{Eqn 4.1})
and in the second term of (\ref{Eqn 4.2}) converges weakly and the another
one converges strongly in the corresponding spaces. This implies that $%
f\left( u_{m}\right) \rightharpoonup f\left( u\right) $ in \ $L_{q}\left(
Q\right) $

Hence, all conditions of Corollary \ref{Corollary 2.1} are fulfilled.\
Applying it to the considered problem (\ref{Eqn 1.1})-(\ref{Eqn 1.3}) we
obtain the statement of Theorem \ref{Theorem 2.1}.

\begin{remark}
\label{Remark 4.1}The solvability theorem such as Theorem \ref{Theorem 2.1}
for the problem (\ref{Eqn 1.1})-(\ref{Eqn 1.3}), but with $u\left(
0,x\right) =u_{0}\left( x\right) $ for $u_{0}\in \overset{0}{S}_{\Delta
,\rho ,2}\left( \Omega \right) $ is also valid and can be proved as in \cite%
{S4} (or \cite{S3}).
\end{remark}

\begin{remark}
\label{Remark 4.2}The problem (\ref{Eqn 1.1})-(\ref{Eqn 1.3}) can also be
considered with the term $b(t,x,u)$ instead of $b_{0}\left\vert u\right\vert
^{\mu +1}$. In this case, it is enough to assume holding of the following
conditions: The function $b\left( t,x,u\right) $ is the Caratheodory
function on $Q\times R^{1}$, there exist functions $b_{0}\left( t,x\right)
,b_{1}\left( t,x\right) \geq 0$ and number $\mu \geq 0$ such that $\min
\left\{ 0,\frac{\rho }{2}-1\right\} \leq \mu <\rho $ and 
\begin{equation*}
\left\vert b\left( t,x,u\right) \right\vert \leq b_{0}\left( t,x\right)
\left\vert u\right\vert ^{\mu +1}+b_{1}\left( t,x\right) ,
\end{equation*}%
where 
\begin{equation*}
b_{0}\in L^{\infty }\left( Q\right) ,\ b_{1}\in L_{2}\left( 0,T;\overset{0}{W%
}\text{\/}_{2}^{1}\left( \Omega \right) \right) \qquad \text{if }\mu \geq 
\frac{\rho }{2}-1;
\end{equation*}

\ $b\in L^{\infty }\left( 0,T;W^{1,\infty }\left( \Omega \right)
;C^{1}\left( R^{1}\right) \right) $ and 
\begin{eqnarray*}
\left\vert D_{i}b\left( t,x,u\right) \right\vert &\leq &\widetilde{b}%
_{0}\left( t,x\right) \left\vert u\right\vert ^{\mu +1}+b_{1}\left(
t,x\right) ,\ i=\overline{1,n}, \\
\left\vert b_{\xi }\left( t,x,\xi \right) \right\vert &\leq &b_{2}\left(
t,x\right) \left\vert \xi \right\vert ^{\mu }+b_{3}\left( t,x\right) ,
\end{eqnarray*}%
\begin{equation*}
\widetilde{b}_{0},b_{2}\in L^{\infty }\left( Q\right) ,\ \widetilde{b}%
_{1}\in L_{2}\left( Q\right) ,\ b_{3}\in L_{2}\left( 0,T;\overset{0}{W}\text{%
\/}_{2}^{1}\left( \Omega \right) \right) ,
\end{equation*}
$q=p^{\prime }=\frac{p}{p-1}$, and $\ p=\rho +2$ if $\ \mu <\frac{\rho }{2}%
-1 $.
\end{remark}

\section{On a Behavior of the Solutions of Problem (\protect\ref{Eqn 1.1})-(%
\protect\ref{Eqn 1.3})}

In this section we investigate the behavior of solutions for different $\mu
\geq 0$: $\min \left\{ 0,\frac{\rho }{2}-1\right\} \leq \mu <\rho $ and $%
u\left( 0,x\right) =u_{0}\left( x\right) $, and in the case $\mu =\rho $.

\begin{theorem}
\label{Theorem 5.1}Let $\min \left\{ 0,\frac{\rho }{2}-1\right\} \leq \mu
<\rho ,$ $u_{0}\in \overset{0}{W}$\/$_{p}^{1}\left( \Omega \right) ,$ $h\in
L^{\infty }\left( R_{+}^{1};L_{q}\left( \Omega \right) \right) $ and $%
\left\Vert h\right\Vert _{L_{q}}\left( t\right) \leq C_{0}$. Then, the
solution of the problem (\ref{Eqn 1.1})-(\ref{Eqn 1.3}) with the initial
condition $u\left( 0,x\right) =u_{0}\left( x\right) $ satisfies the
inequality 
\begin{equation}
\left\Vert u\left( t\right) \right\Vert _{L_{2}\left( \Omega \right)
}^{2}\leq \left( \frac{C+C_{2}\left\Vert h\right\Vert _{L^{\infty }\left(
L_{q}\right) }^{q}}{C_{1}}\right) ^{\frac{2}{p}}+\left( C_{1}\frac{\rho }{2}%
t\right) ^{-\frac{2}{\rho }},  \label{Eqn 5.1}
\end{equation}%
i.e. the solution of the problem (\ref{Eqn 1.1})-(\ref{Eqn 1.3}) remains
bounded as $t\nearrow \infty $, where $C_{j}=C_{j}\left( \rho ,\mu
,b_{0},C_{0},\left\Vert u_{0}\right\Vert _{\overset{0}{W}\/_{p}^{1}},mes~%
\text{\/}\Omega \right) $.
\end{theorem}

\begin{proof}
Consider the functional 
\begin{equation*}
\Phi \left( t\right) \equiv \Phi \left( u\left( t\right) \right) \equiv 
\frac{1}{2}\underset{\Omega }{\dint }\text{\/}\left\vert u\left( t\right)
\right\vert ^{2}dx\equiv \frac{1}{2}\left\Vert u\left( t\right) \right\Vert
_{L_{2}\left( \Omega \right) }^{2}.
\end{equation*}%
If $u\left( t\right) $ is a solution of the problem (\ref{Eqn 1.1})-(\ref%
{Eqn 1.3}) then, the function $\Phi \left( t\right) $ has the property 
\begin{equation*}
\Phi ^{\prime }\left( t\right) =\left\langle u^{\prime },u\right\rangle
=\left\langle \left\vert u\right\vert ^{\rho }\Delta u-b_{0}\left\vert
u\right\vert ^{\mu +1}+h,u\right\rangle =-\left( \rho +1\right) \left\langle
\left\vert u\right\vert ^{\rho }\nabla u,\nabla u\right\rangle -
\end{equation*}%
\begin{equation*}
\left\langle b_{0}\left\vert u\right\vert ^{\mu }u,u\right\rangle
+\left\langle h,u\right\rangle \leq -\left( \rho +1\right) \left\Vert
\left\vert u\right\vert ^{\frac{\rho }{2}}\nabla u\right\Vert
_{L_{2}}^{2}+\left\vert b_{0}\right\vert \left\Vert u\right\Vert _{L_{\mu
+2}}^{\mu +2}+\left\vert \left\langle h,u\right\rangle \right\vert .
\end{equation*}%
Applying the results of Section 3 we get 
\begin{equation*}
\Phi ^{\prime }\left( t\right) \leq -\left( \rho +1\right) \left\Vert
\left\vert \left\vert u\right\vert ^{\frac{\rho }{2}}\nabla u\right\vert
\right\Vert _{L_{2}}^{2}+\left\vert b_{0}\right\vert \left\Vert u\right\Vert
_{L_{\mu +2}}^{\mu +2}+\left\Vert u\right\Vert _{L_{p}}\left\Vert
h\right\Vert _{L_{q}}\leq
\end{equation*}%
\begin{equation*}
-\left( \rho +1\right) \frac{\left( \rho +2\right) ^{2}}{4}\left\Vert
\left\vert \nabla \left( \left\vert u\right\vert ^{\frac{\rho }{2}}u\right)
\right\vert \right\Vert _{L_{2}}^{2}+2\varepsilon \left\Vert u\right\Vert
_{L_{p}}^{p}+C\left( \varepsilon \right) \left( 1+\left\Vert h\right\Vert
_{L_{q}}^{q}\right) \leq
\end{equation*}%
\begin{equation*}
-C_{0}\left\Vert \left\vert \nabla \left( \left\vert u\right\vert ^{\frac{%
\rho }{2}}u\right) \right\vert \right\Vert _{L_{2}}^{2}+C\left( \varepsilon
\right) \left( 1+\left\Vert h\right\Vert _{L_{q}}^{q}\right) \leq
\end{equation*}%
\begin{equation*}
-\widetilde{C}_{0}\left\Vert u\right\Vert _{Lp}^{p}+C\left( \varepsilon
\right) \left( 1+\left\Vert h\right\Vert _{L_{q}}^{q}\right) \leq
-C_{1}\left\Vert u\right\Vert _{L_{2}}^{p}+C\left( \varepsilon \right)
+k\left( \varepsilon \right) \left\Vert h\right\Vert _{L^{\infty }\left(
L_{q}\right) }^{q}
\end{equation*}%
because of $\overset{0}{S}_{1,\rho ,2}\left( \Omega \right) \subset
L_{p}\left( \Omega \right) \subset L_{2}\left( \Omega \right) $ with the
corresponding inequalities. Hence we have 
\begin{equation}
\Phi ^{\prime }\left( t\right) +C_{1}\left( \Phi \left( t\right) \right) ^{%
\frac{p}{2}}\leq C+C_{2}\left\Vert h\right\Vert _{L^{\infty }\left(
L_{q}\right) }^{q},  \label{Eqn 5.2}
\end{equation}%
where $C=C\left( \rho ,\mu ,b_{0},C_{0},\left\Vert u_{0}\right\Vert _{%
\overset{0}{W}\/_{p}^{1}},\ mes~\text{\/}\Omega \right) $, $\
C_{1}=C_{1}\left( \rho ,\mu ,b_{0},C_{0},\left\Vert u_{0}\right\Vert _{%
\overset{0}{W}\/_{p}^{1}},\ mes~\text{\/}\Omega \right) $ and $\Phi \left(
0\right) =\left\Vert u_{0}\right\Vert _{L_{2}}^{2}$.

Then, applying the following form of Gronwall's lemma (Lemma \ref{Lemma 5.1}%
) to the inequality (\ref{Eqn 5.2}), which was proved by Ghidaglia, with $%
y\left( t\right) \equiv \Phi \left( t\right) $, $\theta =C_{1}$, $\eta
=C+C_{2}\left\Vert h\right\Vert _{L_{q}}^{q}$, $l=\frac{p}{2}$ we obtain the
inequality (\ref{Eqn 5.1}).
\end{proof}

\begin{lemma}
\label{Lemma 5.1}(\cite{T}) Let $y\left( t\right) $ be a positive absolutely
continuous function on $R_{+}^{1}$ which satisfies 
\begin{equation*}
y^{\prime }+\theta y^{l}\leq \eta ,\quad l>1,\theta >0,\eta \geq 0.
\end{equation*}%
Then, for $t\geq 0$, 
\begin{equation*}
y\left( t\right) \leq \left( \frac{\eta }{\theta }\right) ^{\frac{1}{l}%
}+\left( \theta \left( l-1\right) t\right) ^{-\frac{1}{l-1}}.
\end{equation*}
\end{lemma}

Now, consider the following problem: 
\begin{equation}
\frac{\partial u}{\partial t}-\left\vert u\right\vert ^{\rho }\Delta
u-b\left( x\right) \left\vert u\right\vert ^{\rho +1}=0,\quad \left(
t,x\right) \in Q,\quad  \label{Eqn 5.3}
\end{equation}%
\begin{equation}
u\left( 0,x\right) =u_{0}\left( x\right) \geq 0,\quad u\left\vert \text{\/~}%
_{\Gamma }=0\right. ,\ \Gamma \equiv \left[ 0,T\right] \times \partial
\Omega ,  \label{Eqn 5.4}
\end{equation}

Let $\lambda _{1}$ be the first eigenvalue and $v_{1}\left( x\right) $ be
the corresponding eigenfunction of the problem 
\begin{equation*}
-\Delta v=\lambda v,\ \ x\in \Omega \quad v\left\vert \text{\/~}_{\partial
\Omega }=0\right. .
\end{equation*}

\begin{lemma}
\label{Lemma 5.2}Let $\rho ,b\left( x\right) >0$, $u_{0}\left( x\right) $ $%
\geq 0$ and $\mu =\rho $, moreover $u_{0}\in L_{2}\left( \Omega \right) $, $%
\left\Vert b\right\Vert _{L^{\infty }\left( \Omega \right) }\leq c$ , $%
c=c\left( \Omega \right) >0$, $\Omega \subset R^{n}$ be as above. Then, if $%
\ M=\frac{c\left( \rho +2\right) ^{2}}{4\left( \rho +1\right) }<\lambda _{1}$
then a solution of the problem (\ref{Eqn 5.3})-(\ref{Eqn 5.4}) remains
bounded as $t\nearrow \infty $ , i.e. the inequality of the type (\ref{Eqn
5.1}) takes place also.
\end{lemma}

\begin{proof}
Let $\rho ,b\left( x\right) >0$ and $\mu =\rho $. Then, using the previous
reasoning we get%
\begin{equation*}
\Phi ^{\prime }\left( t\right) =\left\langle u^{\prime },u\right\rangle
=-\left( \rho +1\right) \left\langle \left\vert u\right\vert ^{\rho }\nabla
u,\nabla u\right\rangle +\left\langle b\left( x\right) \left\vert
u\right\vert ^{\rho +1},u\right\rangle \leq
\end{equation*}%
\begin{equation*}
-\left( \rho +1\right) \frac{4}{\left( \rho +2\right) ^{2}}\left\Vert \nabla
\left( \left\vert u\right\vert ^{\frac{\rho }{2}}u\right) \right\Vert
_{L_{2}}^{2}+\left\Vert b\right\Vert _{L^{\infty }}\left\Vert \left\vert
u\right\vert ^{\frac{\rho }{2}}u\right\Vert _{L_{2}}^{2}.
\end{equation*}%
Since $M<\lambda _{1}$ a solution remains bounded when $t\nearrow \infty $
as in the previous case.
\end{proof}

\begin{remark}
\label{Remark 5.1}Suppose $b\left( x\right) >\lambda _{1}$ and $u\left(
t,x\right) >0$ for $\ x\in \Omega $ or on a subdomain $\widetilde{\Omega }%
\subset \subset \Omega $. Note that this case was studied under various
conditions in \cite{FMcL, Wig, TI, Win}. In our consideration we study the
problem (\ref{Eqn 5.3})-(\ref{Eqn 5.4}) in the following way:

If $u\left( t,x\right) >0$ for $\ x\in \Omega $ then, the equation (\ref{Eqn
5.3}) can be represented as 
\begin{equation*}
u^{-\rho }\frac{\partial u}{\partial t}-\Delta u-b\left( x\right) u=0,\quad
\left( t,x\right) \in Q.\quad
\end{equation*}%
Hence, we have 
\begin{equation*}
\left\langle u^{-\rho }\frac{\partial u}{\partial t},v_{1}\right\rangle
=\left\langle \Delta u+b\left( x\right) u,v_{1}\right\rangle \Longrightarrow
\left\langle u^{-\rho }\frac{\partial u}{\partial t},v_{1}\right\rangle
=-\lambda _{1}\left\langle u,v_{1}\right\rangle +\left\langle b\left(
x\right) u,v_{1}\right\rangle
\end{equation*}%
or 
\begin{equation*}
\left\langle u^{-\rho }\frac{\partial u}{\partial t},v_{1}\right\rangle \geq
\delta \left\langle u,v_{1}\right\rangle ,\quad \left( b\left( x\right)
-\lambda _{1}\right) \geq \delta >0\Longrightarrow \left( 1-\rho \right)
^{-1}\frac{\partial }{\partial t}\left\langle u^{1-\rho },v_{1}\right\rangle
\geq \delta \left\langle u,v_{1}\right\rangle .
\end{equation*}%
The blow-up result can be obtained from here as in \cite{Win} (see \cite%
{FMcL, Wig, TI, Win} and references therein).
\end{remark}

\section{Appendixes}

\subsection{Appendix \textbf{A}}

Let $X,Y$ be a locally convex vector topological spaces, $B\subseteq Y$ be a
Banach space and $g:D\left( g\right) \subseteq X\longrightarrow Y$. Let's
introduce the following subset of $X$ 
\begin{equation*}
\mathcal{M}_{gB}\equiv \left\{ x\in X\left\vert ~g\left( x\right) \in
B,\right. \func{Im}g\cap B\neq \varnothing \right\} .
\end{equation*}

\begin{definition}
\label{Definition 6.1}A subset $\mathcal{M}\subseteq X$ is called a $pn-$%
space (i.e. pseudonormed space) if $S$ is a topological space and there is a
function $\left[ \cdot \right] _{\mathcal{M}}:\mathcal{M}\longrightarrow
R_{+}^{1}\equiv \left[ 0,\infty \right) $ (wh\i ch is called $p-$norm of $%
\mathcal{M}$) such that

qn) $\left[ x\right] _{\mathcal{M}}\geq 0$, $\forall x\in \mathcal{M}$ and $%
x=0\Longrightarrow \left[ x\right] _{\mathcal{M}}=0$;

pn) \ $\left[ x_{1}\right] _{\mathcal{M}}\neq \left[ x_{2}\right] _{\mathcal{%
M}}\Longrightarrow x_{1}\neq x_{2}$, for $x_{1},x_{2}\in \mathcal{M}$, and $%
\left[ x\right] _{\mathcal{M}}=0\Longrightarrow x=0$;
\end{definition}

The following conditions are often fulfilled in the spaces $\mathcal{M}_{gB}$%
.

N) There exist a convex function $\nu :R^{1}\longrightarrow \overline{%
R_{+}^{1}}$ and number $K\in \left( 0,\infty \right] $ such that $\left[
\lambda x\right] _{\mathcal{M}}\leq \nu \left( \lambda \right) \left[ x%
\right] _{\mathcal{M}}$ for any $x\in \mathcal{M}$ and $\lambda \in R^{1}$, $%
\left\vert \lambda \right\vert <K$, moreover, $\underset{\left\vert \lambda
\right\vert \longrightarrow \lambda _{j}}{\lim }\frac{\nu \left( \lambda
\right) }{\left\vert \lambda \right\vert }=c_{j}$, $j=0,1$ where $\lambda
_{0}=0$, $\lambda _{1}=K$ and $c_{0}=c_{1}=1$ or $c_{0}=0$, $c_{1}=\infty $,
i.e. if $K=\infty $ then $\lambda x\in \mathcal{M}$ for any $x\in \mathcal{M}
$ and $\lambda \in R^{1}.$

Let $g:D\left( g\right) \subseteq X\longrightarrow Y$ be such a mapping that 
$\mathcal{M}_{gB}\neq \varnothing $ and the following conditions are
fulfilled

G$_{\text{1}}$) $g:D\left( g\right) \longleftrightarrow \func{Im}g$ is a
bijection and $g\left( 0\right) =0$;

G$_{\text{2}}$) there is a function $\nu :R^{1}\longrightarrow \overline{%
R_{+}^{1}}$ satsfying condition N such that 
\begin{equation*}
\left\Vert g\left( \lambda x\right) \right\Vert _{B}\leq \nu \left( \lambda
\right) \left\Vert g\left( x\right) \right\Vert _{B},\ \forall x\in \mathcal{%
M}_{gB},\ \forall \lambda \in R^{1};
\end{equation*}%
If the mapping $g$ satisfies conditions G$_{1}$ and G$_{2}$ then $\mathcal{M}%
_{gB}$ is a $pn-$space with $p-$norm defined in the following way: there is
a one-to-one function $\psi :R_{+}^{1}\longrightarrow R_{+}^{1}$, $\psi
\left( 0\right) =0$, $\psi ,\psi ^{-1}\in C^{0}$ such that $\left[ x\right]
_{\mathcal{M}_{gB}}\equiv \psi ^{-1}\left( \left\Vert g\left( x\right)
\right\Vert _{B}\right) $. In this case $\mathcal{M}_{gB}$ is a metric space
with a metric: $d_{\mathcal{M}}\left( x_{1};x_{2}\right) \equiv \left\Vert
g\left( x_{1}\right) -g\left( x_{2}\right) \right\Vert _{B}$. Further, we
consider just such type of $pn-$spaces.

\begin{definition}
\label{Definition 6.2}The $pn-$space $\mathcal{M}_{gB}$ is called weakly
complete if $g\left( \mathcal{M}_{gB}\right) $ is weakly closed in $B.$ The
pn-space $\mathcal{M}_{gB}$ is "reflexive" if each bounded weakly closed
subset of $\mathcal{M}_{gB}$ is weakly compact in $\mathcal{M}_{gB}$.
\end{definition}

It is clear that if $B$ is a reflexive Banach space and $\mathcal{M}_{gB}$
is a weakly complete $pn-$space, then $\mathcal{M}_{gB}$ is "reflexive".
Moreover, if $B$ is a separable Banach space, then $\mathcal{M}_{gB}$ is
separable, also.

\subsection{Appendix \textbf{B}}

In the beginning we consider an operator equation 
\begin{equation}
f\left( x\right) =y,\quad y\in Y,  \label{Eqn 6.1}
\end{equation}%
where $f:D\left( f\right) \subseteq X\longrightarrow Y$ is a nonlinear
bounded operator, and prove a general solvability theorem for it. It is
clear that (\ref{Eqn 6.1})$\mathit{\ }$is equivalent to the following
functional equation:%
\begin{equation}
\left\langle f\left( x\right) ,y^{\ast }\right\rangle =\left\langle
y,y^{\ast }\right\rangle ,\quad \forall y^{\ast }\in Y^{\ast }.
\label{Eqn 6.2}
\end{equation}

We consider the following conditions:

1) $f:\mathcal{M}_{0}\subseteq D\left( f\right) \longrightarrow Y$ is a
weakly compact (weakly "continuous") mapping, i.e. for any weakly
convergence sequence $\left\{ x_{m}\right\} _{m=1}^{\infty }\subset \mathcal{%
M}_{0}$ in $\mathcal{M}_{0}$ (i.e. $x_{m}\overset{\mathcal{M}_{0}}{%
\rightharpoonup }x_{0}\in \mathcal{M}_{0}$) there is a subsequence $\left\{
x_{m_{k}}\right\} _{k=1}^{\infty }\subseteq \left\{ x_{m}\right\}
_{m=1}^{\infty }$ such that $f\left( x_{m_{k}}\right) \overset{Y}{%
\rightharpoonup }f\left( x_{0}\right) $ weakly in $Y$ (or for a general
sequence if $\mathcal{M}_{0}$ not is separable space) and $\mathcal{M}_{0}$
be a weakly complete $pn-$space;

2) there exists a mapping $g:X_{0}\subseteq X\longrightarrow Y^{\ast }$ and
a continuous function $\varphi :R_{+}^{1}\longrightarrow R^{1}$
nondecreasing for $\tau \geq \tau _{0}\geq 0$ and $\varphi \left( \tau
_{1}\right) >0$ for a number $\tau _{1}>0$ such that it generates a
"coercive" pair in a generalized sense with $f$ on the topological space $%
X_{1}\subseteq X_{0}\cap \mathcal{M}_{0}$, i.e. 
\begin{equation*}
\left\langle f\left( x\right) ,g\left( x\right) \right\rangle \geq \varphi
\left( \lbrack x]_{\mathcal{M}_{0}}\right) [x]_{\mathcal{M}_{0}},\quad
\forall x\in X_{1},
\end{equation*}%
where $X_{1}$ is such a topological space that $\overline{X_{1}}%
^{X_{0}}\equiv X_{0}$ and $\overline{X_{1}}^{\mathcal{M}_{0}}\equiv \mathcal{%
M}_{0}$, and\textit{\ }$\left\langle \cdot ,\cdot \right\rangle $ is a%
\textit{\ }dual form of the pair $\left( Y,Y^{\ast }\right) $, moreover, one
of the following conditions $\left( \alpha \right) $ or $\left( \beta
\right) $ holds:

$\left( \alpha \right) $ if $g\equiv L$ is a linear continuous operator,
then $X_{1}$ is a \textquotedblright reflexive\textquotedblright\ space (see
[S3, S4]), $X_{0}\equiv X_{1}\subseteq \mathcal{M}_{0}$ is a separable
topological vector space which is dense in $\mathcal{M}_{0}$ and $\ker
L^{\ast }=\left\{ 0\right\} $.

$\left( \beta \right) $ if $g$ is a bounded operator (linear or nonlinear),
then $Y$ is a reflexive separable space, $g\left( X_{1}\right) $ contains an
everywhere dense linear manifold of $Y^{\ast }$ and $g^{-1}$ is weakly
compact (weakly continuous) operator from $Y^{\ast }$ to $\mathcal{M}_{0}$.

\begin{theorem}
\label{Theorem 6.1}\textit{Let conditions 1 and 2 hold. Then the equation (%
\ref{Eqn 6.1}) (or (\ref{Eqn 6.2}))\ is solvable in }$\mathcal{M}_{0}$%
\textit{\ for any} $y\in Y$ \textit{satisfying the following inequality:
there exists} $r>0$ such that\textit{\ } 
\begin{equation}
\varphi \left( \lbrack x]_{\mathcal{M}_{0}}\right) [x]_{\mathcal{M}_{0}}\geq
\left\langle y,g\left( x\right) \right\rangle ,\text{ for}\quad \forall x\in
X_{1}\quad \text{with}\quad \lbrack x]_{\mathcal{M}}\geq r.  \label{Eqn 6.3}
\end{equation}
\end{theorem}

\begin{proof}
Assume that the conditions 1 and 2 ($\alpha $)\textit{\ }are fulfilled and $%
y\in Y$ such that (\ref{Eqn 6.3}) holds. We are going to use Galerkin's
approximation method. Let $\left\{ x^{k}\right\} _{k=1}^{\infty }$ be a
complete system in the (separable) space $X_{1}\equiv X_{0}$. Then, we are
looking for approximate solutions in the form $x_{m}=\overset{m}{\underset{%
k=1}{\sum }}c_{mk}x^{k},$ where $c_{mk}$ are unknown coefficients, that
might be determined from the system of algebraic equations 
\begin{equation}
\Phi _{k}\left( c_{m}\right) :=\left\langle f\left( x_{m}\right) ,g\left(
x^{k}\right) \right\rangle -\left\langle y,g\left( x^{k}\right)
\right\rangle =0,\quad k=1,2,...,m  \label{Eqn 6.4}
\end{equation}%
with $c_{m}\equiv \left( c_{m1},c_{m2},...,c_{mm}\right) $.

We observe that the mapping \ $\Phi \left( c_{m}\right) :=\left( \Phi
_{1}\left( c_{m}\right) ,\Phi _{2}\left( c_{m}\right) ,...,\Phi _{m}\left(
c_{m}\right) \right) $ is continuous by virtue of condition 1. (\ref{Eqn 6.4}%
) implies the existence of such $r=r\left( \left\Vert y\right\Vert
_{Y}\right) >0$ that the \textquotedblright acute angle\textquotedblright\
condition is fulfilled for all $x_{m}$ with $\left[ x_{m}\right] _{\mathcal{M%
}_{0}}\geq r$, i.e. for any $c_{m}\in S_{r_{1}}^{R^{m}}\left( 0\right)
\subset R^{m}$, $r_{1}\geq r$ the inequality 
\begin{equation*}
\overset{m}{\underset{k=1}{\sum }}\left\langle \Phi _{k}\left( c_{m}\right)
,c_{mk}\right\rangle \equiv \left\langle f\left( x_{m}\right) ,g\left( 
\overset{m}{\underset{k=1}{\sum }}c_{mk}x^{k}\right) \right\rangle
-\left\langle y,g\left( \overset{m}{\underset{k=1}{\sum }}c_{mk}x^{k}\right)
\right\rangle =\quad
\end{equation*}%
\begin{equation*}
\left\langle f\left( x_{m}\right) ,g\left( x_{m}\right) \right\rangle
-\left\langle y,g\left( x_{m}\right) \right\rangle \geq 0,\quad \forall
c_{m}\in 
\mathbb{R}
^{m},\left\Vert c_{m}\right\Vert _{%
\mathbb{R}
^{m}}=r_{1}.
\end{equation*}%
holds. The solvability of system (\ref{Eqn 6.4}) for each $m=1,2,\ldots $
follows from a well-known lemma on the \textquotedblleft acute
angle\textquotedblright\ (\cite{L, D, S3}),\ which is equivalent to the
Brouwer's fixed-point theorem. Thus, the sequence $\left\{ x_{m}\left\vert
~m\geq \right. 1\right\} $ of the approximate solutions, that is contained
in a bounded subset of the space $\mathcal{M}_{0}$. Further arguments are
analogous to those from \cite{L, S4} therefore we omit them. It remains to
pass to the limit in (\ref{Eqn 6.4})\ by $m$ and use a weak convergency of a
subsequence of the sequence $\left\{ x_{m}\left\vert ~m\geq \right.
1\right\} $, the weak compactness of the mapping $f$, and finally, the
completeness of the system $\left\{ x^{k}\right\} _{k=1}^{\infty }$in the
space $X_{1}$.

Hence, we get the limit element $x_{0}=w-\underset{j\nearrow \infty }{\lim }%
x_{m_{j}}\in \mathcal{M}_{0}$ that is a solution of the equation 
\begin{equation}
\left\langle f\left( x_{0}\right) ,g\left( x\right) \right\rangle
=\left\langle y,g\left( x\right) \right\rangle ,\quad \forall x\in X_{0},
\label{Eqn 6.5}
\end{equation}%
or 
\begin{equation}
\left\langle g^{\ast }\circ f\left( x_{0}\right) ,x\right\rangle
=\left\langle g^{\ast }\circ y,x\right\rangle ,\quad \forall x\in X_{0}.
\label{Eqn 6.6}
\end{equation}

In the second case, i.e. when the conditions 1 and 2 ($\beta $)\ are
fulfilled and $y\in Y$ such that (\ref{Eqn 6.3}) holds, the approximate
solutions suppose to be looked for in the form 
\begin{equation}
x_{m}=g^{-1}\left( \overset{m}{\underset{k=1}{\sum }}c_{mk}y_{k}^{\ast
}\right) \equiv g^{-1}\left( y_{\left( m\right) }^{\ast }\right) ,\quad
i.e.\ x_{m}=g^{-1}\left( y_{\left( m\right) }^{\ast }\right)  \label{Eqn 6.7}
\end{equation}%
where $\left\{ y_{k}^{\ast }\right\} _{k=1}^{\infty }\subset Y^{\ast }$ is a
complete system in the (separable) space $Y^{\ast }$ and belongs to $g\left(
X_{1}\right) $. The unknown coefficients $c_{mk}$, might be determined from
the system of algebraic equations 
\begin{equation}
\widetilde{\Phi }_{k}\left( c_{m}\right) :=\left\langle f\left( x_{m}\right)
,y_{k}^{\ast }\right\rangle -\left\langle y,y_{k}^{\ast }\right\rangle
=0,\quad k=1,2,...,m  \label{Eqn 6.8}
\end{equation}%
with $c_{m}\equiv \left( c_{m1},c_{m2},...,c_{mm}\right) $. Taking this and
our conditions into account we get 
\begin{equation}
\left\langle f\left( x_{m}\right) ,y_{k}^{\ast }\right\rangle -\left\langle
y,y_{k}^{\ast }\right\rangle =\left\langle f\left( g^{-1}\left( y_{\left(
m\right) }^{\ast }\right) \right) ,y_{k}^{\ast }\right\rangle -\left\langle
y,y_{k}^{\ast }\right\rangle =0,  \label{Eqn 6.9}
\end{equation}%
for $k=1,2,...,m$.

As it was observed above the mapping \ 
\begin{equation*}
\widetilde{\Phi }\left( c_{m}\right) :=\left( \widetilde{\Phi }_{1}\left(
c_{m}\right) ,\widetilde{\Phi }_{2}\left( c_{m}\right) ,...,\widetilde{\Phi }%
_{m}\left( c_{m}\right) \right)
\end{equation*}%
is continuous by virtue of the conditions 1 and 2($\beta $). Also, (\ref{Eqn
6.3}) implies the existence of such $\widetilde{r}>0$ that the
\textquotedblright acute angle\textquotedblright\ condition is fulfilled for
all $y_{\left( m\right) }^{\ast }$ with $\left\Vert y_{\left( m\right)
}^{\ast }\right\Vert _{Y^{\ast }}\geq \widetilde{r}$ , i.e. for any $%
c_{m}\in S_{r_{1}}^{R^{m}}\left( 0\right) \subset R^{m}$, $\widetilde{r}%
_{1}\geq \widetilde{r}$ the inequality 
\begin{equation*}
\overset{m}{\underset{k=1}{\sum }}\left\langle \widetilde{\Phi }_{k}\left(
c_{m}\right) ,c_{mk}\right\rangle \equiv \left\langle f\left( x_{m}\right) ,%
\overset{m}{\underset{k=1}{\sum }}c_{mk}y_{k}^{\ast }\right\rangle
-\left\langle y,\overset{m}{\underset{k=1}{\sum }}c_{mk}y_{k}^{\ast
}\right\rangle =\quad
\end{equation*}%
\begin{equation*}
\left\langle f\left( g^{-1}\left( y_{\left( m\right) }^{\ast }\right)
\right) ,y_{\left( m\right) }^{\ast }\right\rangle -\left\langle y,y_{\left(
m\right) }^{\ast }\right\rangle =\left\langle f\left( x_{m}\right) ,g\left(
x_{m}\right) \right\rangle -\left\langle y,g\left( x_{m}\right)
\right\rangle \geq 0,
\end{equation*}%
\begin{equation*}
\forall c_{m}\in 
\mathbb{R}
^{m},\left\Vert c_{m}\right\Vert _{%
\mathbb{R}
^{m}}=\widetilde{r}_{1}.
\end{equation*}%
holds by virtue of our conditions. Consequently, the solvability of system (%
\ref{Eqn 6.8}) (or (\ref{Eqn 6.9})) for each $m=1,2,\ldots $ follows from
the \textquotedblleft acute angle\textquotedblright\ lemma as above. Thus,
we obtained a sequence $\left\{ y_{\left( m\right) }^{\ast }\left\vert
~m\geq \right. 1\right\} $ of the approximate solutions, that is contained
in a bounded subset of $Y^{\ast }$. This implies an existence of a
subsequence $\left\{ y_{\left( m_{j}\right) }^{\ast }\right\} _{j=1}^{\infty
}$ that convergences weakly in $Y^{\ast }$. Consequently, the sequence $%
\left\{ x_{m_{j}}\right\} _{j=1}^{\infty }\equiv \left\{ g^{-1}\left(
y_{\left( m_{j}\right) }^{\ast }\right) \right\} _{j=1}^{\infty }$ converges
weakly in the space $\mathcal{M}_{0}$ by vertue of the condition 2($\beta $)
(may be after passing to the subsequence). It remains to pass to the limit
in (\ref{Eqn 6.9})\ by $j$ and use a weak convergency of the subsequence of
the sequence $\left\{ y_{\left( m\right) }^{\ast }\left\vert ~m\geq \right.
1\right\} $, the weak compactness of mappings $f$ and $g^{-1}$, and the
completeness of the system $\left\{ y_{k}^{\ast }\right\} _{k=1}^{\infty }$%
in the space $Y^{\ast }$.

Hence, we get a limit element $x_{0}=w-\underset{j\nearrow \infty }{\lim }%
x_{m_{j}}$ $=w-\underset{j\nearrow \infty }{\lim }g^{-1}\left( y_{\left(
m_{j}\right) }^{\ast }\right) \in \mathcal{M}_{0}$ that is the solution of
the equation 
\begin{equation}
\left\langle f\left( x_{0}\right) ,y^{\ast }\right\rangle =\left\langle
y,y^{\ast }\right\rangle ,\quad \forall y^{\ast }\in Y^{\ast }.
\label{Eqn 6.10}
\end{equation}%
Q.e.d.
\end{proof}

\begin{remark}
\label{Remark 6.1}\textit{It is obvious, that if there exists a function }$%
\psi :R_{+}^{1}\longrightarrow R_{+}^{1}$, $\psi \in C^{0}$\textit{such that 
}$\psi \left( \xi \right) =0\Longleftrightarrow \xi =0$\textit{\ and the
inequality }$\psi \left( \left\Vert x_{1}-x_{2}\right\Vert _{X}\right) \leq
\left\Vert f\left( x_{1}\right) -f\left( x_{2}\right) \right\Vert _{Y}$%
\textit{\ \ is fulfilled for all }$x_{1},x_{2}\in \mathcal{M}_{0},$\textit{\
\ then a solution of the equation (\ref{Eqn 6.2}) is unique. }
\end{remark}

\begin{corollary}
\label{Corollary 6.1}Assume that the conditions of Theorem 7 are fulfilled
and \textit{there is a continuous function }$\varphi
_{1}:R_{+}^{1}\longrightarrow R_{+}^{1}$ such that $\left\Vert g\left(
x\right) \right\Vert _{Y^{\ast }}\leq \varphi _{1}\left( [x]_{\mathcal{M}%
_{0}}\right) $ for any $x\in X_{0}$ and\textit{\ }$\varphi \left( \tau
\right) \nearrow +\infty $\textit{\ }and $\frac{\varphi \left( \tau \right)
\tau }{\varphi _{1}\left( \tau \right) }\nearrow +\infty $\textit{\ as }$%
\tau \nearrow +\infty $. \textit{Then, } the \textit{equation (\ref{Eqn 6.2}%
) is solvable in }$\mathcal{M}_{0}$ \textit{for any }$y\in Y$.
\end{corollary}

\subsection{Appendix \textbf{C }}

Now, we are ready to provide the proof of Theorem \ref{Theorem 2.2}. Let $%
\left\{ x^{k}\right\} _{k=1}^{\infty }$ be a complete system in the
(separable) space $X_{0}$ and $\left\{ \theta ^{s}\left( t\right) \right\}
_{s=1}^{\infty }$ be a complete system in the (separable) space $L_{p}\left(
0,T\right) ,$ then $\left\{ \theta ^{s}\left( t\right) x^{k}\right\}
_{s,k=1}^{\infty }$ is a complete system in the separable space $L_{p}\left(
0,T;X_{0}\right) $.

\textbf{Proof of the Theorem \ref{Theorem 2.2}.} We are going to use the
method of elliptic regularization (see, for example, \cite{L}\footnote{%
see, also, Soltanov K. N., Sprekels J. - Nonlinear equations in nonreflexive
Banach spaces and fully nonlinear equations, Advances in Mathematical
Sciences and Applications, 1999, v. 9, no. 2, 939-972.}). Namely, first we
prove the solvability of the following auxiliary elliptic problem with a
small parameter $\varepsilon >0$. 
\begin{equation}
-\varepsilon \frac{d^{2}x_{\varepsilon }}{dt^{2}}+\frac{dx_{\varepsilon }}{dt%
}+f(t,x_{\varepsilon }\left( t\right) )=y\left( t\right) ,  \label{Eqn 6.11}
\end{equation}%
\begin{equation}
x_{\varepsilon }\left( 0\right) =0,\quad \frac{dx_{\varepsilon }}{dt}%
\left\vert ~_{t=T}\right. =0,~\varepsilon >0.  \label{Eqn 6.12}
\end{equation}

A solution of the problem (\ref{Eqn 6.11})-(\ref{Eqn 6.12}) would be
understood as an element $x_{\varepsilon }\left( t\right) \in $ $\underset{0}%
{\mathbf{P}}$\/$_{1,p,q}\left( 0,T;\mathcal{M}_{0},Y\right) $ that satisfies
the following functional equation 
\begin{equation*}
\varepsilon \underset{0}{\overset{T}{\int }}\left\langle \frac{%
dx_{\varepsilon }}{dt},\frac{dy^{\ast }}{dt}\right\rangle dt+\underset{0}{%
\overset{T}{\int }}\left\langle \frac{dx_{\varepsilon }}{dt},y^{\ast
}\right\rangle dt+
\end{equation*}%
\begin{equation}
\underset{0}{\overset{T}{\int }}\left\langle f(t,x_{\varepsilon }\left(
t\right) ),y^{\ast }\right\rangle dt=\underset{0}{\overset{T}{\int }}%
\left\langle y,y^{\ast }\right\rangle dt  \label{Eqn 6.13}
\end{equation}%
for any $y^{\ast }\in W_{q%
{\acute{}}%
}^{1}\left( 0,T;Y^{\ast }\right) \cap \left\{ y^{\ast }\left( t\right)
\left\vert ~y^{\ast }\left( 0\right) =0\right. \right\} $.

\begin{lemma}
\label{Lemma 6.1}Under the conditions of Theorem \ref{Theorem 2.2} the
equation (\ref{Eqn 6.13}) is solvable in the space $\underset{0}{\mathbf{P}}$%
\/$_{1,p,q}\left( 0,T;\mathcal{M}_{0},Y\right) $ for any $y\in G$ where $G$
is defined in Theorem \ref{Theorem 2.2}.
\end{lemma}

The statement of this lemma follows from Theorem \ref{Theorem 6.1} of
Appendix B (see, also \cite{S3, S5}). Indeed, the mapping generated by the
considered problem (\ref{Eqn 6.11})-(\ref{Eqn 6.12}) is weakly compact from $%
\underset{0}{\mathbf{P}}$\/$_{1,p,q}\left( 0,T;\mathcal{M}_{0},Y\right) $
into $\left( W_{q%
{\acute{}}%
}^{1}\left( 0,T;Y^{\ast }\right) \cap \left\{ y^{\ast }\left( t\right)
\left\vert ~y^{\ast }\left( 0\right) =0\right. \right\} \right) ^{\ast }$ by
virtue of condition (\textit{ii})\textit{\ }and because of first two terms
are linear bounded operators. Moreover, inequalities 
\begin{equation*}
\varepsilon \underset{0}{\overset{T}{\int }}\left\langle \frac{%
dx_{\varepsilon }}{dt},L\frac{dx_{\varepsilon }}{dt}\right\rangle dt+%
\underset{0}{\overset{T}{\int }}\left\langle \frac{dx_{\varepsilon }}{dt}%
,Lx_{\varepsilon }\right\rangle dt+
\end{equation*}%
\begin{equation*}
\underset{0}{\overset{T}{\int }}\left\langle f(t,x_{\varepsilon }\left(
t\right) ),Lx_{\varepsilon }\right\rangle dt-\underset{0}{\overset{T}{\int }}%
\left\langle y,Lx_{\varepsilon }\right\rangle dt\geq \varepsilon
C_{0}\left\Vert \frac{dx_{\varepsilon }}{dt}\right\Vert _{L_{q}\left(
0,T;Y\right) }^{\nu }+
\end{equation*}%
\begin{equation*}
\varphi \left( \left[ x_{\varepsilon }\right] _{L_{p}\left( \mathcal{M}%
_{0}\right) }\right) \left[ x_{\varepsilon }\right] _{L_{p}\left( \mathcal{M}%
_{0}\right) }-c_{1}\left\Vert y\right\Vert _{L_{q}\left( Y\right) }\left[
x_{\varepsilon }\right] _{L_{p}\left( \mathcal{M}_{0}\right) }-\left(
1+\varepsilon \right) C_{2}\geq
\end{equation*}%
\begin{equation}
\left[ \varphi \left( \left[ x_{\varepsilon }\right] _{L_{p}\left( 0,T;%
\mathcal{M}_{0}\right) }\right) -c_{1}\left\Vert y\right\Vert _{L_{q}\left(
0,T;Y\right) }\right] \left[ x_{\varepsilon }\right] _{L_{p}\left( 0,T;%
\mathcal{M}_{0}\right) }-c  \label{Eqn 6.14}
\end{equation}%
are fulfilled for any $x_{\varepsilon }\in W_{p}^{1}\left( 0,T;X_{0}\right)
\cap \left\{ x_{\varepsilon }\left( t\right) \left\vert ~x_{\varepsilon
}\left( 0\right) =0\right. \right\} .$It is also clear that for a
sufficiently large $p-$norm of $x_{\varepsilon }\left( t\right) $ there is a
subset of $L_{q}\left( 0,T;Y\right) $ such that the last expression in (\ref%
{Eqn 6.14}) is greater than zero under the conditions of Theorem \ref%
{Theorem 2.2}. These and conditions \textit{iii }and\textit{\ iv }show%
\textit{\ }that the other conditions of the above mentioned result are
fulfilled and, consequently, the equation (\ref{Eqn 6.13}) is solvabile (see
also \cite{S5}). Thus, for each $y\in L_{q}\left( 0,T;Y\right) $ there is a
function $x_{\varepsilon }\in \underset{0}{\mathbf{P}}$\/$_{1,p,q}\left( 0,T;%
\mathcal{M}_{0},Y\right) $ that satisfies the equation (\ref{Eqn 6.13}) for
any $\forall y^{\ast }\in W_{q\prime }^{1}\left( 0,T;Y^{\ast }\right) $ i.e. 
\begin{equation*}
\varepsilon \underset{0}{\overset{T}{\int }}\left\langle \frac{%
dx_{\varepsilon }}{dt},\frac{dy^{\ast }}{dt}\right\rangle dt+\underset{0}{%
\overset{T}{\int }}\left\langle \frac{dx_{\varepsilon }}{dt},y^{\ast
}\right\rangle dt+
\end{equation*}%
\begin{equation}
+\underset{0}{\overset{T}{\int }}\left\langle f(t,x_{\varepsilon }\left(
t\right) ),y^{\ast }\right\rangle dt=\underset{0}{\overset{T}{\int }}%
\left\langle y,y^{\ast }\right\rangle dt,\quad \forall y^{\ast }\in \overset{%
0}{W}~_{q\prime }^{1}\left( 0,T;Y^{\ast }\right) .  \label{Eqn 6.15}
\end{equation}%
The equality (\ref{Eqn 6.15}) can be rewritten in the form 
\begin{equation*}
\varepsilon \underset{0}{\overset{T}{\int }}\left\langle \frac{%
dx_{\varepsilon }}{dt},\frac{dy^{\ast }}{dt}\right\rangle dt=\underset{0}{%
\overset{T}{\int }}\left\langle y-\frac{dx_{\varepsilon }}{dt}%
-f(t,x_{\varepsilon }\left( t\right) ),y^{\ast }\right\rangle dt
\end{equation*}%
where $y-\frac{dx_{\varepsilon }}{dt}-f(t,x_{\varepsilon }\left( t\right) )$
belongs to $L_{q}\left( 0,T;Y\right) $ because of $y\in L_{q}\left(
0,T;Y\right) $ and $\frac{dx_{\varepsilon }}{dt}$, $f(t,x_{\varepsilon
}\left( t\right) )\in L_{q}\left( 0,T;Y\right) $ for any $x_{\varepsilon
}\in \underset{0}{\mathbf{P}}$\/$_{1,p,q}\left( 0,T;\mathcal{M}_{0},Y\right) 
$. Hence, according to our conditions, for each fixed $\varepsilon >0,$ and
boundedness of the right part of (\ref{Eqn 6.14}) we obtain $\frac{%
d^{2}x_{\varepsilon }}{dt^{2}}\in L_{q}\left( 0,T;Y\right) $ and
consequently, the boundary condition $\frac{dx_{\varepsilon }}{dt}\left\vert
~_{t=T}\right. $ is defined properly. Thus, the function $x_{\varepsilon
}\left( t\right) $ is a solution of the equation 
\begin{equation*}
-\varepsilon \underset{0}{\overset{T}{\int }}\left\langle \frac{%
d^{2}x_{\varepsilon }}{dt^{2}},y^{\ast }\right\rangle dt+\underset{0}{%
\overset{T}{\int }}\left\langle \frac{dx_{\varepsilon }}{dt},y^{\ast
}\right\rangle dt+
\end{equation*}%
\begin{equation}
+\underset{0}{\overset{T}{\int }}\left\langle f(t,x_{\varepsilon }\left(
t\right) ),y^{\ast }\right\rangle dt=\underset{0}{\overset{T}{\int }}%
\left\langle y,y^{\ast }\right\rangle dt,\quad \forall y^{\ast }\in
L_{q\prime }\left( 0,T;Y^{\ast }\right) .  \label{Eqn 6.16}
\end{equation}%
On the other hand, considering this equation for any $y^{\ast }\in
W~_{q\prime }^{1}\left( 0,T;Y^{\ast }\right) $ and comparing it with (\ref%
{Eqn 6.15}) we obtain $\frac{dx_{\varepsilon }}{dt}\left\vert ~_{t=T}\right.
=0$ (by using argumentations similar to those from \cite{L, D}).

Hence, we proved that the problem (\ref{Eqn 6.11})-(\ref{Eqn 6.12}) is
solvable in the space $\underset{0}{\mathbf{P}}$\/$_{1,p,q}\left( 0,T;%
\mathcal{M}_{0},Y\right) $ $\cap $ $W~_{q\prime }^{2}\left( 0,T;Y\right)
\cap $ $\left\{ x_{\varepsilon }\left( t\right) \left\vert ~\frac{%
dx_{\varepsilon }}{dt}\left\vert ~_{t=T}\right. =0\right. \right\} $ for
each fixed $\varepsilon >0$.

Now, it is necessary to pass to the limit at $\varepsilon \searrow 0$. To
this end, we need the uniformly on $\epsilon $ estimation of $\frac{%
dx_{\varepsilon }}{dt}$.

Further, we consider the equation 
\begin{equation*}
-\varepsilon \underset{0}{\overset{T}{\int }}\frac{d^{2}}{dt^{2}}%
\left\langle x_{\varepsilon },Lx^{k}\right\rangle \theta ^{s}\left( t\right)
dt+\underset{0}{\overset{T}{\int }}\frac{d}{dt}\left\langle x_{\varepsilon
},Lx^{k}\right\rangle \theta ^{s}\left( t\right) dt=
\end{equation*}%
\begin{equation}
\underset{0}{\overset{T}{\int }}\left\langle y\left( t\right)
-f(t,x_{\varepsilon }\left( t\right) ),Lx^{k}\right\rangle \theta ^{s}\left(
t\right) dt\equiv \underset{0}{\overset{T}{\int }}\left\langle
f_{0}(t,x_{\varepsilon }\left( t\right) ),Lx^{k}\right\rangle \theta
^{s}\left( t\right) dt,  \label{Eqn 6.17}
\end{equation}%
where $\left\{ \theta ^{s}\left( t\right) x^{k}\right\} _{s,k=1}^{\infty }$
is a complete system in $W_{p}^{1}\left( 0,T;X_{0}\right) ,$ and
consequently, $\frac{dx_{\varepsilon }}{dt}$ is a solution of the problem 
\begin{equation}
-\varepsilon \frac{d^{2}x_{\varepsilon }}{dt^{2}}+\frac{dx_{\varepsilon }}{dt%
}=f_{0\varepsilon }(t),\quad t\in \left( 0,T\right)  \label{Eqn 6.18}
\end{equation}%
\begin{equation}
x_{\varepsilon }\left( 0\right) =0,\quad \frac{dx_{\varepsilon }}{dt}%
\left\vert ~_{t=T}\right. =0,  \label{Eqn 6.19}
\end{equation}%
as $x_{\varepsilon }\left( t\right) $ belongs to a bounded subset of $%
L_{p}\left( 0,T;\mathcal{M}_{0}\right) $ and $f_{0}(t,x_{\varepsilon }\left(
t\right) )$ belongs to a bounded subset of $L_{q}\left( 0,T;Y\right) $ at $%
\varepsilon \searrow 0$, consequently $f_{0\varepsilon }\in L_{q}\left(
0,T\right) $ and belongs to a bounded subset of this space at $\varepsilon
\searrow 0$ and under the conditions of $y\left( t\right) $.

The solution of the problem (\ref{Eqn 6.18})-(\ref{Eqn 6.19}) satisfies 
\begin{equation*}
\frac{dx_{\varepsilon }\left( t\right) }{dt}=\frac{1}{\varepsilon }\underset{%
0}{\overset{T-t}{\int }}f_{0\varepsilon }\left( T-\tau \right) ~exp\left\{ -%
\frac{T-t-\tau }{\varepsilon }\right\} d\tau .
\end{equation*}%
Applying the generalized Minkowski's inequality and taking into account that 
$\frac{1}{\varepsilon }\overset{\infty }{\underset{0}{\int }}~exp\left\{ -%
\frac{\tau }{\varepsilon }\right\} d\tau =1$ we get $\left\Vert \frac{%
dx_{\varepsilon }}{dt}\right\Vert _{L_{q}\left( 0,T;Y\right) }\leq C<\infty $
for some positive $C$ that is undependent on $\varepsilon $. Thus, for each $%
y\left( t\right) \in L_{q}\left( 0,T;Y\right) $ the function $x_{\varepsilon
}(t)$ belongs to a bounded subset of the space$\ \underset{0}{\mathbf{P}\/}%
\/_{1,p,q}\left( 0,T;\mathcal{M}_{0},Y\right) $ uniformly on $\varepsilon $.
The "reflexivity" of $\mathcal{M}_{0}$ and the reflexivity of $Y$ allow us
to pass to the limit for $\varepsilon $ $\searrow 0$ in all terms of the (%
\ref{Eqn 6.17}) except for the first one. Therefore, it remains to estimate
just the first term of (\ref{Eqn 6.17}). We have 
\begin{equation*}
\left\vert -\varepsilon \underset{0}{\overset{T}{\int }}\left\langle \frac{%
d^{2}x_{\varepsilon }}{dt^{2}},y^{\ast }\right\rangle dt\right\vert \leq
\varepsilon \underset{0}{\overset{T}{\int }}\left\vert \left\langle \frac{%
dx_{\varepsilon }}{dt},\frac{dy^{\ast }}{dt}\right\rangle \right\vert dt\leq
\end{equation*}%
\begin{equation*}
\varepsilon \left\Vert \frac{dx_{\varepsilon }}{dt}\right\Vert _{L_{q}\left(
0,T;Y\right) }\left\Vert \frac{dy^{\ast }}{dt}\right\Vert _{L_{q\prime
}\left( 0,T;Y^{\ast }\right) }
\end{equation*}%
for any $y^{\ast }\in W_{q%
{\acute{}}%
}^{1}\left( 0,T;Y^{\ast }\right) \cap \left\{ y^{\ast }\left( t\right)
\left\vert ~y^{\ast }\left( 0\right) =0\right. \right\} $. Taking into
account the estimation $\varepsilon \left\Vert \frac{dx_{\varepsilon }}{dt}%
\right\Vert _{L_{q}\left( 0,T;Y\right) }^{\nu }$ $\leq C<\infty $ for $\nu
>1 $, that is valid by virtue of the a priori estimations, we get 
\begin{equation*}
\left\vert -\varepsilon \underset{0}{\overset{T}{\int }}\left\langle \frac{%
d^{2}x_{\varepsilon }}{dt^{2}},y^{\ast }\right\rangle dt\right\vert \leq
\varepsilon ^{\frac{\nu -1}{\nu }}\widetilde{C}\left\Vert \frac{dy^{\ast }}{%
dt}\right\Vert _{L_{q\prime }\left( 0,T;Y^{\ast }\right) }.
\end{equation*}%
This means that the first term of (\ref{Eqn 6.17}) vanishes when $%
\varepsilon \searrow 0$. Thus, considering the equation (\ref{Eqn 6.17}) for
any $\xi \in L_{p}\left( 0,T;X_{0}\right) ,$ passing to the limit at $%
\varepsilon \searrow 0$ and taking into account that $\ker L^{\ast }=\left\{
0\right\} $ complete the proof of Theorem \ref{Theorem 2.2}.

\bigskip

\end{document}